\documentclass[12pt,a4paper]{article}

\usepackage{amsfonts}
\usepackage{amsmath}
\usepackage{amsbsy}
\usepackage{theorem}
\usepackage{graphicx}
\usepackage{amssymb}
\usepackage{epstopdf}

\newtheorem{theorem}{Theorem}
\newtheorem{proposition}{Proposition}
\newtheorem{lemma}{Lemma}

\newtheorem{remark}{Remark}

\begin{document}
\sf
\title{\normalsize \bfseries SPECTRAL ANALYSIS OF LONG RANGE DEPENDENCE IN FUNCTIONAL TIME SERIES
}
\date{}
\author{\normalsize M. Dolores Ruiz--Medina\\ \normalsize  IMAG - Unidad de excelencia María de Maeztu - CEX2020-001105-M\\ \normalsize  University of Granada}
\maketitle

\begin{abstract}

Long  Range Dependence (LRD) in functional
sequences is characterized  in the spectral domain under suitable conditions. Particularly,
 multifractionally
integrated functional autoregressive moving averages processes can be introduced in this framework.
The  convergence to zero in
the Hilbert-Schmidt operator norm of the integrated bias of the periodogram operator is proved.
Under a Gaussian scenario, a  weak--consistent parametric
estimator of the long--memory operator is then obtained by minimizing, in  the  norm of bounded linear operators, a divergence information functional  loss.
\end{abstract}

\medskip


\section{Introduction}

One can find evidence of LRD in time series data arising in several
areas like   agriculture, environment, economics, finance,
geophysics, just to mention a few.
 Indeed, a
huge literature on this topic  has been developed over the last few
decades (c.f., \cite{Beran94};  \cite{Beran13};  \cite{Giraitis12};  \cite{Palma07}; \cite{Robinson03}).  This framework allows the
description of processes with long persistence in time. In the
stationary case, LRD is characterized by a slow decay of the
covariance function, and an unbounded spectral density, typically at
zero frequency. In the real--valued process framework, we refer to
the reader to the papers \cite{Andrews}; \cite{Anh04};   \cite{Gao02};  \cite{Gao01};
\cite{Giraitis90}; \cite{Hosoya97}; \cite{Leonenko06};
\cite{Sun03}, among others.

The analysis of  LRD phenomena in an infinite--dimensional  process framework is a
challenging topic where several problems remain open.
 Only a few contributions can be found on this topic in functional
  time series analysis.  LRD in functional sequences is characterized
  by the  non--summability in time of
the nuclear norms of the associated  family of covariance operators. In the
linear case,
 a variable--order
fractional power law  usually characterizes the asymptotic behavior in time   of the norms of the involved    bounded linear operator
 sequence. That is the case of  the
approaches
 in the current literature based   on operator--valued
processes.
A fractional Brownian motion with values in a Hilbert
space, involving an operator-valued Hurst coefficient, is considered
in \cite{Rackauskasv2} (see also \cite{Rackauskas} on the functional
analytical tools applied).
 In \cite{duker18},  a central and functional
central limit theorems are obtained under non--summability of the  operator
norm  sequence. The limit process in this functional central limit
result  is a self--similar process, characterized by an operator  defining the self--similarity  exponent. Note that the LRD
models introduced in these papers in the linear setting are
characterized and analyzed in the time domain. Recently, in
\cite{duker20}, for LRD  linear processes in a separable Hilbert space,
 a stochastic--integral based approach is adopted to
representing the limiting process of the sample autocovariance
operator in the space of Hilbert--Schmidt operators.

 A semiparametric linear framework
 has been
adopted to analyze  LRD in functional sequences in
\cite{LiRobinsonShang19}. The functional
dependence structure is specified via the projections of the curve
process onto different sub-spaces, spanned by the eigenvectors  of
the long-run covariance function. A Central Limit Theorem is derived
under suitable regularity conditions. Functional Principal Component
Analysis  is applied in the consistent estimation of the orthonormal
functions spanning the dominant subspace, where the projected curve
process displays the largest dependence range.    The memory parameter and the
dimension of the dominant subspace are estimated as well.
 The conditions assumed are satisfied, in particular, by a functional
version of fractionally integrated autoregressive moving averages
processes.  Some interesting applications to US stock prices and age
specific fertility rates are also provided.

As follows from the above cited references, the spectral domain  has
not been exploited yet in the formulation and  estimation of LRD in stationary
functional time series. Furthermore, LRD functional time series models have mainly been introduced in the
linear setting.
 Our paper
attempts to cover these gaps.   To this aim, the spectral representation of a self--adjoint operator on a separable Hilbert space, in terms of a
spectral family of projection operators, is considered. Suitable conditions are then assumed on the symbol defining such a representation, for
the spectral density operator family at a neighborhood of zero frequency. Specifically,  the  behavior of the spectral density operator at zero frequency is characterized by a bounded symmetric positive operator family, whose operator norm slowly varies   at zero  frequency, composed with an unbounded operator at zero frequency involving the  long--memory operator.
   The corresponding covariance operator family  displays a  heavy tail behavior in time as proved in Proposition \ref{prlrd}.  As an interesting
special case, we  refer to a  family of
fractionally integrated functional autoregressive
moving averages processes of variable order  (see also Remark 9 in \cite{LiRobinsonShang19}). Several additional examples can be found by  tapering, in the frequency domain,  the symbols of the spectral density operator family, associated with
infinite--dimensional stationary LRD processes in continuous time. Particularly, we consider the case of fractional integration of variable order of functional processes with rational spectral density operator (see, e.g., \cite{AnhLeonenkoa};  \cite{AnhLeonenkob};  \cite{Kelbert05}).
 The   convergence to zero, in the Hilbert--Schmidt operator norm,
 of the integrated periodogram bias operator is derived, under  the square integrability in the frequency
domain of the Hilbert--Schmidt operator norm of the  spectral
density operator family.  This condition holds under mild conditions, in our case,   under  the  second--order property of the functional process, assuming  the  integrability in the frequency domain of the operator norm of the spectral density operator family.
 The weak consistency of the proposed parametric estimator of the long--memory operator then follows in the Gaussian case, extending Theorem 3 in \cite{Anh04}.

 Note that the parametric estimation approach in the spectral domain has not been exploited yet in the functional time series context.  Under short--range dependence
 (SRD),  \cite{Panaretos13}  adopts a nonparametric framework. Specifically, a weighted average  of the functional values of the periodogram operator
 is considered as an  estimator of the spectral density operator.
  This methodology is not
 applicable when one wants to approximate the behavior of the spectral density operator at zero frequency
 in the presence of LRD. In  this paper, we consider a  parametric estimator of the long--memory operator,  computed by minimizing the operator norm of a weighted Kullback--Leibler
divergence operator. This operator compares the behavior at a neighborhood of zero frequency  of the  true spectral
density operator, underlying to the curve data, with  the possible
semiparametric candidates. On the other hand, this functional is
linear with respect to the periodogram operator. This is an
important  advantage of the proposed estimation methodology in
relation to nonparametric kernel estimation.

   The outline of the paper is the following.
   Preliminary definitions,  results and first conditions are established in  Section \ref{prem}. The main assumptions are formulated in Section \ref{LRDlfs}.
   Under this setting of conditions, LRD is characterized in the functional spectral domain.
       The heavy tail behavior in time of the associated covariance operator family is obtained in Proposition \ref{prlrd}. Some examples are provided as well.  In Section \ref{sec2}, the convergence to zero of the Hilbert-Schmidt operator norm of the integrated bias of the periodogram operator is proved
       in Theorem \ref{pr1}.
Under a   Gaussian scenario,
Theorem
\ref{theslrd} in Section \ref{sththeslrd} derives the  consistent parametric estimation  of the long--memory operator in the functional spectral domain.    Some final comments  can be found in Section \ref{conclus}.

\section{Preliminaries}

\label{prem}
    In what follows,  $(\Omega,\mathcal{A},P)$  denotes the
basic probability space. Let  $H$ be a real separable Hilbert space with the inner product $\left\langle\cdot,\cdot\right\rangle_{H}.$  Denote by $\widetilde{H}=H+iH,$ its complex version
 whose elements are functions of the form
$$\psi=\varphi_{1}+ i\varphi_{2},\quad \varphi_{i}\in H,\ i=1,2,$$
\noindent with the inner product
\begin{equation}
\left\langle\varphi_{1}+ i\varphi_{2},\phi_{1}+i\phi_{2}\right\rangle_{\widetilde{H}}=\left\langle \varphi_{1},\phi_{1}\right\rangle_{H}
+\left\langle \varphi_{2},\phi_{2}\right\rangle_{H}+i\left(\left\langle \varphi_{2},\phi_{1}\right\rangle_{H}-\left\langle \varphi_{1},\phi_{2}\right\rangle_{H}\right).
\label{ip}
\end{equation}

Recall that
$\mathcal{L}^{2}_{\widetilde{H}}(\Omega,\mathcal{A},P)$ denotes the
space of second--order zero--mean $\widetilde{H}$--valued random
variables on $(\Omega,\mathcal{A},P),$ with the norm
$\|X\|^{2}_{\mathcal{L}^{2}_{\widetilde{H}}(\Omega,\mathcal{A},P)}=
E[\|X\|^{2}_{\widetilde{H}}],$ for every $X\in
\mathcal{L}^{2}_{\widetilde{H}}(\Omega,\mathcal{A},P).$

In the following, fix an orthonormal basis $\left\{ \varphi_{k},\
k\geq 1\right\}$ of $H,$ and consider \begin{equation}\left\{
\psi_{k}=(1/2)\left[\varphi_{k}+i\varphi_{k}\right],\ k\geq
1\right\},\label{ip0}
\end{equation} \noindent  as an orthonormal basis of  $\widetilde{H}.$
All the subsequent identities involving
operator norms can be expressed in terms of such an  orthonormal
basis, allowing the interpretation of $H$ as a closed subspace of
$\widetilde{H}.$  Particularly, the nuclear
      $\left\|\cdot\right\|_{L^{1}(\widetilde{H})},$
and the Hilbert--Schmidt
$\left\|\cdot\right\|_{\mathcal{S}(\widetilde{H})}$ operator
     norms   on $\widetilde{H}$  are defined as follows:
\begin{eqnarray}\left\|\mathcal{A}\right\|_{L^{1}(\widetilde{H})}&=&
\sum_{k\geq 1}\left\langle \left[\mathcal{A}^{\star}
\mathcal{A}\right]^{1/2}(\psi_{k}),\psi_{k}
\right\rangle_{\widetilde{H}}
\nonumber\\
\left\|\mathcal{A}\right\|_{\mathcal{S}(\widetilde{H})} &=&
\left[\sum_{k\geq 1}\left\langle
\mathcal{A}^{\star}\mathcal{A}(\psi_{k}),\psi_{k}
\right\rangle_{\widetilde{H}}\right]^{1/2}= \sqrt{
\left\|\mathcal{A}^{\star}\mathcal{A}\right\|_{L^{1}(\widetilde{H})}},
\label{HSN}\end{eqnarray}
\noindent with $\{\psi_{k},\ k\geq 1\}$ being  an orthonormal  basis of $\widetilde{H}$ as given in (\ref{ip0}).

We denote by $\|\cdot\|_{\mathcal{L}(\widetilde{H})}$     the  norm
in the space of bounded linear operators on $\widetilde{H},$ i.e.,
 $\|\mathcal{A}\|_{\mathcal{L}(\widetilde{H})}=
 \sup_{\psi\in \widetilde{H};\ \|\psi\|=1}
 \|\mathcal{A}(\psi)\|_{\widetilde{H}}.$ This norm is also usually referred
 as the operator norm  (or uniform operator norm).  Through the paper we consider the equality between operators on $\widetilde{H}$  (respectively, on $H$) in
 the norm of the space $\mathcal{L}(\widetilde{H})$ (respectively, of the space $\mathcal{L}(H)$)  implying the pointwise identity of such operators over the functions on $\widetilde{H}$ (respectively on $H$). Otherwise, the norm with respect to which the identity considered holds is established.

 For simplicity of
 notation, in the subsequent development, the letter $\mathcal{K}$
 will refer to a positive  constant whose specific value  may vary from
one to another  inequality or identity.

Let  $\{X_{t},\ t\in \mathbb{Z}\}$ be a strictly stationary
functional time series with zero mean $E\left[X_{t}\right]=0,$ and functional variance  $\sigma_{X}^{2}=E[\|X_{t}\|^{2}_{H}]=
E[\|X_{0}\|^{2}_{H}]=\|R_{0}\|_{L^{1}(H)},$ for every $t\in \mathbb{Z}.$ Also,
\begin{eqnarray}&&
\mathcal{R}_{t}=E[X_{s+t}\otimes X_{s}]=\quad
\forall t,s\in \mathbb{Z}\label{covstatfunct}\\&&
\mathcal{R}_{t}(g)(h)= E[X_{s+t}(h)X_{s}(g)]= E\left[\left\langle
X_{s+t},h\right\rangle_{H}\left\langle X_{s},g\right\rangle_{H}
\right],\quad \forall h,g\in H.\nonumber\\
\label{covstatfunct2}
\end{eqnarray}

     Note that,
  $E[\|X_{0}\|^{2}_{H}]<\infty $ implies
  $P[X_{t}\in H]=1,$
  for all $t\in \mathbb{Z}.$

  Let  $\mathcal{F}_{\omega}$ be the spectral density operator
  on $\widetilde{H},$
  defined by the following identity in the
  $\mathcal{L}(\widetilde{H})$ norm, for $\omega \in [-\pi,\pi]\backslash \{0\}:$
  \begin{equation}\mathcal{F}_{\omega}
  \underset{\mathcal{L}(\widetilde{H})}{=}
  \frac{1}{2\pi} \sum_{t\in \mathbb{Z}}\exp\left(-i\omega t\right)
\mathcal{R}_{t}.\label{sdo2}\end{equation}

\begin{remark} In \cite{Panaretos13},   convergence of
series  (\ref{sdo2}) holds  in the nuclear norm  for SRD functional sequences.  Here, a weaker convergence is assumed.
Indeed,   identity (\ref{sdo2}) could hold  for $\omega \in
[-\pi,\pi]\backslash \Lambda_{0},$ where  $\int_{\Lambda_{0}}d\omega =0.$  In our case,
$\Lambda_{0}= \{0\}$ for the characterization of LRD in  \textbf{Assumption II} below.
\end{remark}

For simplicity, in the following, we will omit the reference to the
set $[-\pi,\pi]\backslash \Lambda_{0},$ when the identities  hold
almost surely in the frequency domain. That is the case of the identities for a spectral density
operator family involving an unbounded spectral density operator   at
zero--frequency  (see equation  (\ref{sosd}) below).

 The functional Discrete Fourier Transform  (fDFT)   $\widetilde{X}^{(T)}$  of the
functional data $\{X_{t},\ t=1,\dots,T\}$ is defined as
\begin{equation}\widetilde{X}^{(T)}_{\omega }(\cdot)
\underset{\widetilde{H}}{=} \frac{1}{\sqrt{2\pi
T}}\sum_{t=1}^{T}X_{t}(\cdot )\exp\left(-i\omega t\right), \quad
\omega\in [-\pi ,\pi],\label{fDFT}\end{equation}\noindent
 where $\underset{\widetilde{H}}{=}$ denotes the equality
in $\widetilde{H}$ norm.  Hence, $\widetilde{X}^{(T)}_{\omega }$ is
$2\pi$--periodic
 and Hermitian with  respect to  $\omega
\in [-\pi ,\pi].$
\begin{remark}\label{rem2}
Under the condition  $E[\|X_{0}\|^{2}_{\widetilde{H}}]<\infty,$ applying triangle inequality,
$$E\left[\|\widetilde{X}_{\omega}^{(T)}\|_{\widetilde{H}}\right]\leq \frac{1}{\sqrt{2\pi
T}}\sum_{t=1}^{T}E\|X_{t}(\cdot )\|_{\widetilde{H}}<\infty,$$ \noindent for every $\omega \in [-\pi,\pi].$
  The fDFT $\widetilde{X}_{\omega}^{(T)}$ defines a random element
   in $\widetilde{H},$  and   $P\left[ \widetilde{X}^{(T)}_{\omega }(\cdot)
   \in \widetilde{H}\right]=1.$ Hence, $\mathcal{F}_{\omega }^{(T)}=
   E\left[\widetilde{X}_{\omega}^{(T)}\otimes
   \overline{\widetilde{X}_{\omega}^{(T)}}\right]\in L^{1}(\widetilde{H}),$
   for $\omega \in [-\pi,\pi].$
\end{remark}

 The periodogram operator
$p_{\omega }^{(T)}=\widetilde{X}_{\omega}^{(T)}\otimes
\overline{\widetilde{X}_{\omega}^{(T)}} $ is  an
empirical operator,
 with mean
 $E[p_{\omega }^{(T)}]=E[\widetilde{X}_{\omega}^{(T)}
 \otimes \widetilde{X}_{-\omega}^{(T)}]=\mathcal{F}_{\omega }^{(T)}.$
 Particularly, under (\ref{sdo2}), for any $T\geq 2,$
 the following  identity holds  in $\mathcal{L}(\widetilde{H}):$
\begin{eqnarray}&&
\mathcal{F}_{\omega }^{(T)}= E\left[p_{\omega
}^{(T)}\right]=\frac{1}{2\pi
T}\left[\sum_{t=1}^{T}\sum_{s=1}^{T}\exp\left(-i\omega
(t-s)\right)E[X_{t}\otimes X_{s}]\right]\nonumber\\
&&=\frac{1}{2\pi }\sum_{u=-(T-1)}^{T-1}\exp\left(-i\omega u\right)
\frac{(T-|u|)}{T}\mathcal{R}_{u}.
\label{meanperiodogram}
\end{eqnarray}

 Let $F_{T}$ be the
F\'ejer kernel, given by
\begin{equation}F_{T}(\omega )=\frac{1}{T}\sum_{t=1}^{T}
\sum_{s=1}^{T}\exp\left(-i(t-s)\omega \right),\quad \omega \in
[-\pi,\pi],\quad T\geq 2.\label{eqfkd}
\end{equation}
 Applying  Fourier Transform Inversion Formula in the space
 $\mathcal{L}(\widetilde{H}),$ from
equations (\ref{meanperiodogram}) and  (\ref{eqfkd}), for each
$\omega \in [-\pi,\pi],$
\begin{eqnarray}&&\hspace*{-2cm} \mathcal{F}_{\omega }^{(T)}
= \left[F_{T}*\mathcal{F}_{\bullet
}\right](\omega )\nonumber\\ &&=
\int_{-\pi}^{\pi} F_{T}(\omega - \xi)
\mathcal{F}_{\xi} d\xi,\quad T\geq 2.
 \label{eqconv}
\end{eqnarray}

\subsection{Preliminaries on spectral analysis of self--adjoint operators}
\label{smo}

This section presents some preliminary elements on spectral theory of self--adjoint operators on a separable Hilbert space  (see, e.g., \cite{Dautray85},
 pp. 112--140).

 It is well--known that, for a self--adjoint operator $\mathcal{D}$ on a separable Hilbert space
$\widetilde{H},$ there exists a family of projection operators $\{E_{\lambda
},\ \lambda \in \Lambda \subseteq \mathbb{R}\},$ also called the spectral family of $\mathcal{D},$
such that the following identity holds:
\begin{equation}
\mathcal{D}=\int_{\Lambda }\lambda dE_{\lambda }.
\label{idws}
\end{equation}
   This family of projection operators satisfies the following properties:
\begin{itemize}
\item[(i)] $E_{\lambda }E_{\mu }=E_{\mbox{inf}\{\lambda ,\mu\}};$
 \item[(ii)] $\lim_{\widetilde{\lambda }\to \lambda;\ \widetilde{\lambda }
 \geq \lambda }E_{\widetilde{\lambda }}=E_{\lambda };$
\item[(iii)]
 $\lim_{\lambda\to -\infty}E_{\lambda }=0;$ $\lim_{\lambda\to \infty}
 E_{\lambda }=I_{\widetilde{H}},$
 \noindent where  $I_{\widetilde{H}}$ denotes the identity operator on $\widetilde{H}.$
\item[(iv)] The domain of $\mathcal{D}$ is defined as
\begin{equation}\mbox{Dom}(\mathcal{D})=\left\{h\in \widetilde{H}:\ \int_{\Lambda }|\lambda |^{2} d\left\langle E_{\lambda }(h),h\right\rangle_{\widetilde{H}}<\infty \right\}.\label{idws3}
\end{equation}
\item[(v)] A  function $G(\mathcal{D})$ admits the representation
\begin{equation}G(\mathcal{D})=\int_{\Lambda }G(\lambda )dE_{\lambda },\label{idws2}
\end{equation}
\noindent and $$\mbox{Dom}(G(\mathcal{D}))=\left\{h\in \widetilde{H}:\ \int_{\Lambda }\left|G(\lambda )\right|^{2} d\left\langle E_{\lambda }(h),h\right\rangle_{\widetilde{H}}<\infty \right\}.$$

The operator integrals (\ref{idws})  and (\ref{idws2})  are understood as improper operator Stieltjes integrals which converge strongly (see, e.g., Section 8.2.1 in  \cite{Ramm05}). Let $\Delta =(a,b],$ $-\infty<a<b<\infty,$  $E_{\Delta }:= E_{b}-E_{a}.$

The family $E_{\Delta }$ of self--adjoint bounded non--negative  operators from the set of Borel sets $\Delta \subseteq \mathbb{R}$ into the space $\mathcal{L}(\widetilde{H})$  of bounded linear operators on a Hilbert space $\widetilde{H}$ is called  an operator measure if $$E_{\left[\cup_{j=1}^{\infty}\Delta_{j}\right]}=\sum_{j=1}^{\infty}E_{\Delta_{j}},$$
\noindent where the limit at the right--hand side is understood  in the sense of weak--convergence of operators, with $\Delta_{i}\cap \Delta_{j}=\varnothing,\ i\neq j,\ E_{\varnothing}=0.
$
\end{itemize}

From (iii),  for every $g,h\in \widetilde{H},$
 \begin{eqnarray}&&\int_{\Lambda }d\left\langle E_{\lambda }(g), h
 \right\rangle_{\widetilde{H}}=\left\langle g,h\right\rangle_{\widetilde{H}}
 \nonumber\\
&&\int_{\Lambda }d\left\langle E_{\lambda }(h), h
 \right\rangle_{\widetilde{H}}=     \left\|h\right\|_{\widetilde{H}}^{2}.\label{initwofunc}
 \end{eqnarray}
 \noindent  Thus, $\{E_{\lambda },\
 \lambda \in \Lambda \subseteq
 \mathbb{R}\}$ provides a resolution of the identity.

Note that,  from (\ref{idws}) (see (i)--(v)),  for  every $\psi \in \mbox{Dom}(\mathcal{D})\subseteq \widetilde{H},$
 \begin{eqnarray}
    \|\mathcal{D}(\psi)\|_{\widetilde{H}}^{2}&=& \left\langle \mathcal{D}(\psi),\mathcal{D}(\psi)\right\rangle_{\widetilde{H}}
 =\left\langle \mathcal{D}\mathcal{D}(\psi),\psi\right\rangle_{\widetilde{H}}=\int_{\Lambda }|\lambda |^{2}d\left\langle E_{\lambda }(\psi),\psi\right\rangle_{\widetilde{H}}\nonumber\\
 &=&\left\langle \left[\mathcal{D}\left(\int_{\Lambda }\lambda d E_{\lambda }\right)\right](\psi),\psi\right\rangle_{\widetilde{H}}
 =\left\langle \left[\left(\int_{\Lambda }\lambda d E_{\lambda }\right)\mathcal{D}\right](\psi),\psi\right\rangle_{\widetilde{H}}. \nonumber\\
 \label{ip}
 \end{eqnarray}

If  $\mathcal{D}\in \mathcal{L}(\widetilde{H}),$  hence,  equation  (\ref{idws3}) holds  for every
   $\psi \in \widetilde{H},$ and from (\ref{ip}),
\begin{eqnarray}
& &    \left\|\left[\mathcal{D}-\int_{\Lambda }\lambda d E_{\lambda }\right](\psi)\right\|_{\widetilde{H}}^{2}=\left\langle \mathcal{D}(\psi)-\int_{\Lambda }\lambda d E_{\lambda }(\psi), \mathcal{D}(\psi)-\int_{\Lambda }\lambda d E_{\lambda }(\psi)\right\rangle_{\widetilde{H}}\nonumber\\
&&\hspace*{-1cm}=\int_{\Lambda }|\lambda |^{2}d\left\langle E_{\lambda }(\psi),\psi\right\rangle_{\widetilde{H}}+\int_{\Lambda }|\lambda |^{2}d\left\langle E_{\lambda }(\psi),\psi\right\rangle_{\widetilde{H}}
-2\int_{\Lambda }|\lambda |^{2}d\left\langle E_{\lambda }(\psi),\psi\right\rangle_{\widetilde{H}}
=0.\nonumber\\
 \label{bon2}
    \end{eqnarray}
\noindent Thus,  $\left\|\mathcal{D}-\int_{\Lambda }\lambda d E_{\lambda }\right\|_{\mathcal{L}(\widetilde{H})}=0,$
and  the weak--sense representation (\ref{idws}) also holds in $\mathcal{L}(\widetilde{H})$--norm.
In particular, for $\mathcal{D}\in
L_{0}(\widetilde{H}),$ with $L_{0}(\widetilde{H})$ denoting the
class of compact operators on  $\widetilde{H},$ the mapping $\lambda
\longrightarrow E_{\lambda}$ has discontinuities at the points given
by the eigenvalues $\{\lambda_{k}(\mathcal{D}),\ k\geq 1\},$ with
$$E_{\lambda_{k}}-\lim_{\widetilde{\lambda}\to
\lambda_{k}(\mathcal{D});\
\widetilde{\lambda}<\lambda_{k}(\mathcal{D})}E_{\widetilde{\lambda}}=
P_{k},$$ \noindent where $P_{k}$ is the projection operator onto the eigenspace
generated by the eigenvectors associated with the eigenvalue
$\lambda_{k}(\mathcal{D}),$ for every $k\geq 1.$

Let now consider the following assumption:

\medskip

\noindent \textbf{Assumption I}. Assume that
\begin{equation}
\int_{-\pi}^{\pi}\left\|\mathcal{F}_{\omega }\right\|_{\mathcal{L}(\widetilde{H})}d\omega <\infty.\label{lhni}\end{equation}

\medskip

\begin{remark} \textbf{Assumption I} holds, for instance, when  the family  $\{ \mathcal{F}_{\omega },\ \omega \in [-\pi,\pi]\}$ is a.s. continuous in $\omega \in [-\pi,\pi],$ with respect to $\mathcal{L}(\widetilde{H})$--norm, since applying  reverse triangle inequality,  $\|\mathcal{F}_{\omega }\|_{\mathcal{L}(\widetilde{H})}$ is  a.s. continuous  in $\omega \in [-\pi,\pi].$

\end{remark}
\begin{remark}
\label{rem1b}
Note that, under \textbf{Assumption I},  for every $t\in \mathbb{Z},$
\begin{eqnarray}
&&\left\|\mathcal{R}_{t}\right\|_{\mathcal{L}(\widetilde{H})}
=\left\|\int_{-\pi}^{\pi}\exp\left( i\omega t\right)\mathcal{F}_{\omega }d\omega \right\|_{\mathcal{L}(\widetilde{H})}\nonumber\\
&&\hspace*{2cm}\leq \int_{-\pi}^{\pi}\left\|\mathcal{F}_{\omega }\right\|_{\mathcal{L}(\widetilde{H})}d\omega <\infty.
\label{eqbonsdo}
\end{eqnarray}

\end{remark}
 The next preliminary result will be applied in the subsequent development.

\begin{lemma}
\label{lempo}  Under \textbf{Assumption I},
 \begin{eqnarray}&&\sum_{t\in
\mathbb{Z}}\|\mathcal{R}_{t}\|_{\mathcal{S}(\widetilde{H})}^{2}
 =\int_{-\pi}^{\pi}\left\|\mathcal{F}_{\omega
}\right\|_{\mathcal{S}(\widetilde{H})}^{2}d\omega <\infty .
\label{eqfcpr1}
\end{eqnarray}
\end{lemma}

\begin{proof}

Given an orthonormal basis $\{\psi_{k},\ k\geq 1\}$ of $\widetilde{H},$
under  \textbf{Assumption I},    $\mathcal{F}_{\omega }$ is  a.s. bounded  in $\omega \in [-\pi,\pi].$ In particular,
  \begin{eqnarray}&&\int_{-\pi}^{\pi}\left\langle\mathcal{F}_{\omega
}(\psi_{k}),\psi_{l}\right\rangle_{\widetilde{H}}d\omega \leq \|\psi_{l}\|_{\widetilde{H}}\int_{-\pi}^{\pi}\left\|\mathcal{F}_{\omega
}(\psi_{k})\right\|_{\widetilde{H}}d\omega \nonumber\\
&&\hspace*{2cm} \leq \int_{-\pi}^{\pi}\left\|\mathcal{F}_{\omega
}\right\|_{\mathcal{L}(\widetilde{H})}d\omega <\infty,\label{fico}\end{eqnarray}\noindent   for every $k,l\geq 1.$  Hence, from (\ref{sdo2}) and (\ref{fico}), for every $t\in \mathbb{Z},$
\begin{equation}\int_{-\pi}^{\pi}\exp\left(it\omega \right)\left\langle\mathcal{F}_{\omega
}(\psi_{k}),\psi_{l}\right\rangle_{\widetilde{H}}d\omega =\left\langle\mathcal{R}_{t}(\psi_{k}),\psi_{l}\right\rangle_{\widetilde{H}}, \quad k,l\geq 1. \label{lhni2}\end{equation}

From (\ref{lhni2}), applying Fourier transform inversion formula, \begin{eqnarray}&& \sum_{t\in
\mathbb{Z}}\|\mathcal{R}_{t}\|_{\mathcal{S}(H)}^{2}=\sum_{t\in
\mathbb{Z}}\sum_{k,l\geq 1}|\mathcal{R}_{t}(\psi_{k})(\psi_{l})|^{2}\nonumber\\
&&=\sum_{t\in
\mathbb{Z}}\sum_{k,l\geq 1}\int_{-\pi}^{\pi}\int_{-\pi}^{\pi}\exp\left(it(\omega-\xi)\right)\mathcal{F}_{\omega}(\psi_{k})(\psi_{l})
\overline{\mathcal{F}_{\xi}(\psi_{k})(\psi_{l})} d\xi d\omega\nonumber \\ &&=\sum_{k,l\geq 1}\int_{-\pi}^{\pi}\int_{-\pi}^{\pi}\left[\sum_{t\in
\mathbb{Z}}\exp\left(it(\omega-\xi)\right)\right]\mathcal{F}_{\omega}(\psi_{k})(\psi_{l})
\overline{\mathcal{F}_{\xi}(\psi_{k})(\psi_{l})} d\xi d\omega\nonumber\\
&&=\sum_{k,l\geq 1}\int_{-\pi}^{\pi}\int_{-\pi}^{\pi}\delta (\omega-\xi)\mathcal{F}_{\omega}(\psi_{k})(\psi_{l})
\overline{\mathcal{F}_{\xi}(\psi_{k})(\psi_{l})} d\xi d\omega \nonumber\\
&&=\int_{-\pi}^{\pi}\sum_{k,l\geq 1}\left|\mathcal{F}_{\omega}(\psi_{k})(\psi_{l})\right|^{2}d\omega= \int_{-\pi}^{\pi}\left\|\mathcal{F}_{\omega
}\right\|_{\mathcal{S}(\widetilde{H})}^{2}d\omega.
\label{eqfcpr1}\end{eqnarray}

From equations   (\ref{lhni2})   and   (\ref{eqfcpr1}), keeping in mind that $\mathcal{F}_{\omega
}$ is nonnegative symmetric operator,
\begin{eqnarray}&& \sum_{t\in
\mathbb{Z}}\|\mathcal{R}_{t}\|_{\mathcal{S}(H)}^{2}=
 \int_{-\pi}^{\pi}\left\|\mathcal{F}_{\omega
}\right\|_{\mathcal{S}(\widetilde{H})}^{2}d\omega=\int_{-\pi}^{\pi}\left\|\mathcal{F}_{\omega
}\mathcal{F}_{\omega
}\right\|_{L^{1}(\widetilde{H})}d\omega
\nonumber\\
&&
\leq
\int_{-\pi}^{\pi}\left\|\mathcal{F}_{\omega
}\right\|_{L^{1}(\widetilde{H})}d\omega
=\int_{-\pi}^{\pi}\sum_{k\geq 1}\left\langle \left[\mathcal{F}_{\omega
}^{\star}\mathcal{F}_{\omega
}\right]^{1/2}(\psi_{k}),\psi_{k}\right\rangle_{\widetilde{H}}d\omega
\nonumber\\
&&=\sum_{k\geq 1}\int_{-\pi}^{\pi}\left\langle\mathcal{F}_{\omega
}(\psi_{k}),\psi_{k}\right\rangle_{\widetilde{H}}d\omega =\sum_{k\geq 1}\left\langle\mathcal{R}_{0}(\psi_{k}),\psi_{k}\right\rangle_{\widetilde{H}}
\nonumber\\
&&=\|\mathcal{R}_{0}\|_{L^{1}(\widetilde{H})}=\sigma^{2}_{X}<\infty.
\label{eqfcpr12}\end{eqnarray}

\end{proof}

\begin{remark}
 \label{re000}
Under \textbf{Assumption I}, from Lemma \ref{lempo},
 $\mathcal{F}_{\omega }\in \mathcal{S}(\widetilde{H}),$
  for $\omega \in [-\pi ,\pi]\backslash\Lambda _{0},$
  with, as before,  $\int_{\Lambda _{0}}d\omega =0.$
     Also, $\left\|\mathcal{F}_{\omega }\right\|_{\mathcal{S}
  (\widetilde{H})}\in L^{2}([-\pi,\pi], \mathbb{C}).$

\end{remark}

\section{Spectral analysis of LRD functional time series}
\label{LRDlfs}

As commented in the Introduction, the literature on
LRD modeling in functional sequences has been mainly
developed in the time domain, under  the context of linear processes
in Hilbert spaces
 (see, e.g., \cite{duker18}; \cite{LiRobinsonShang19}; \cite{Rackauskas};
 \cite{Rackauskasv2}), paying
 special attention  to the theory of
 operator self--similar processes (see
 \cite{Charac14}; \cite{Laha62}; \cite{Lamperti62}; \cite{Matache06},
 among others).

 The next condition characterizes the unbounded behavior at zero frequency of the spectral density operator family.

 \medskip

 \noindent \textbf{Assumption II}. Let $\left\{ \mathcal{A}_{\theta },\ \theta \in \Theta \right\}$
  be a   parametric  family  of positive bounded self--adjoint long--memory operators,
   with $\Theta $ denoting the parameter space.
 For each $\theta\in \Theta,$ assume that    as $\omega \to 0$:
 \begin{equation}
 \left\|\mathcal{F}_{\omega ,\theta }
  |\omega|^{\mathcal{A}_{\theta }}\mathcal{M}_{\omega, \mathcal{F}}^{-1}-I_{\widetilde{H}}\right\|_{\mathcal{L}(\widetilde{H})}\to 0,
  \label{sosd}
\end{equation}
\noindent where  $I_{\widetilde{H}}$ denotes the identity operator on  $\widetilde{H},$ and
$\{\mathcal{M}_{\omega, \mathcal{F}},\ \omega \in [-\pi,\pi]\}$ is a  family of  bounded positive self--adjoint  operators.

 \medskip

 For $\omega \in [-\pi,\pi]$ and $\theta \in \Theta ,$ the spectral representation of $\mathcal{M}_{\omega, \mathcal{F}}$ and  $\mathcal{A}_{\theta }$  in terms of a common spectral family  $\left\{E_{\lambda },\ \lambda \in \Lambda \right\}$ of projection operators (see Section \ref{smo}) is considered in the next assumption.

     \medskip

     \noindent \textbf{Assumption III}.   Assume that  $\mathcal{A}_{\theta}$ and $\mathcal{M}_{\omega, \mathcal{F}}$
   admit the following spectral representations:
    \begin{eqnarray}
    \mathcal{A}_{\theta} &\underset{\mathcal{L}(\widetilde{H})}{=}&\int_{\Lambda }\alpha (\lambda , \theta)d E_{\lambda },\quad \theta \in \Theta \nonumber\\
    \mathcal{M}_{\omega, \mathcal{F}}&\underset{\mathcal{L}(\widetilde{H})}{=}&\int_{\Lambda }M_{\omega , \mathcal{F}}(\lambda )
    d E_{\lambda },\quad  \omega \in [-\pi,\pi].
    \label{ssfdo}
    \end{eqnarray}

  We refer to $\{\alpha (\lambda ,\theta ),\ \lambda \in \Lambda\}$
and $\{M_{\omega, \mathcal{F} }(\lambda ),\ \lambda \in \Lambda\}$ as the
respective symbols of the self--adjoint operators
$\mathcal{A}_{\theta }$ and $\mathcal{M}_{\omega , \mathcal{F} }.$

    \medskip

\begin{remark}\label{remcofso}
   From (\ref{ssfdo}), operators $|\omega |^{-\mathcal{A}_{\theta }}$
   and $\mathcal{M}_{\omega ,\mathcal{F}}$ commute,  for  any  $\theta  \in \Theta ,$
   and $\omega \in [-\pi,\pi]\backslash \Lambda_{0}$
     (see, e.g., \cite{Dautray85}).  \end{remark}

Under   \textbf{Assumptions II--III},  as $\omega \to 0,$

\begin{eqnarray}&&
\left\|\mathcal{F}_{\omega ,\theta
}\int_{\Lambda }\frac{|\omega
|^{\alpha (\lambda ,\theta )}}{M_{\omega ,\mathcal{F}}(\lambda )}dE_{\lambda}-I_{\widetilde{H}}\right\|_{\mathcal{L}(\widetilde{H})}\to 0
,\quad  \forall  \theta \in \Theta.
  \label{sfasdo}
 \end{eqnarray}

\medskip

\noindent \textbf{Assumption IV}.    $\{\alpha (\lambda ,\theta ),\ \lambda \in \Lambda\}$
and $\{M_{\omega, \mathcal{F} }(\lambda ),\ \lambda \in \Lambda\}$ satisfy:
\begin{itemize}
\item[(i)]
  For   $(\lambda ,\theta )\in \Lambda\times \Theta,$   there exist $l_{\alpha }(\theta )$ and
$ L_{\alpha }(\theta )$ such that
\begin{eqnarray}&& 0< l_{\alpha }(\theta )
 \leq  \alpha (\lambda ,\theta )\leq  L_{\alpha }(\theta )< 1\nonumber\\
&& l_{\alpha }(\theta )=
  \inf_{\psi\in \widetilde{H};\
  \|\psi\|_{\widetilde{H}}=1}\left\langle \mathcal{A}_{\theta }(\psi),
   \psi\right\rangle_{\widetilde{H}},\
   L_{\alpha }(\theta )=
   \sup_{\psi\in \widetilde{H};\ \|\psi\|_{\widetilde{H}}=1}
   \left\langle \mathcal{A}_{\theta }(\psi),\psi
   \right\rangle_{\widetilde{H}}.\nonumber\\ \label{disa}
 \end{eqnarray}
\item[(ii)]  For each $\lambda_{0} \in \Lambda ,$ $ M_{\omega, \mathcal{F} }(\lambda_{0} )$ is slowly varying function at $\omega =0$
in the Zygmund's sense (see, e.g., Definition 6.6 in \cite{Beran17}). Furthermore, we also assume
that there   exist positive constants $m$ and $M$ such that, for every $\omega \in [-\pi,\pi],$
\begin{eqnarray}&& m\leq
M_{\omega ,\mathcal{F} }(\lambda )\leq M,\quad \forall \lambda \in \Lambda\label{dbnpp}\\
&&  m<\inf_{\psi\in \widetilde{H};\
   \|\psi\|_{\widetilde{H}}=1}\left\langle \mathcal{M}_{\omega ,\mathcal{F} }(\psi),
   \psi\right\rangle_{\widetilde{H}}<  \sup_{\psi\in \widetilde{H};\
    \|\psi\|_{\widetilde{H}}=1}\left\langle \mathcal{M}_{\omega ,\mathcal{F} }(\psi),
    \psi\right\rangle_{\widetilde{H}}<M.\nonumber
\end{eqnarray}

\end{itemize}

     \subsection{LRD characterization in the time domain}
     \label{ilrdchsd}
     The next proposition  shows the heavy--tail behavior in time
       of the inverse functional Fourier transform of the  spectral density operator family satisfying the
     above formulated conditions.
\begin{proposition}\label{prlrd}
Let $ \left\{\mathcal{F}_{\omega ,\theta }, \ (\omega ,\theta )\in
[-\pi,\pi]\times \Theta \right\}$  be the semiparametric family of
spectral density operators satisfying \textbf{Assumptions I--IV}. Consider
\begin{equation}\mathcal{R}_{t,\theta
}\underset{\mathcal{L}(\widetilde{H})}{=}
\int_{-\pi}^{\pi}\exp\left(i\omega t\right)\mathcal{F}_{\omega
,\theta }d\omega ,\quad t\in \mathbb{Z},\ \theta \in
\Theta.\label{sdo2b}\end{equation} \noindent    Then,
\begin{eqnarray}
&&\left\|\mathcal{R}_{t,\theta
}\left[\widetilde{\mathcal{M}}_{t,\mathcal{F},\mathcal{A}_{\theta
}}t^{\mathcal{A}_{\theta
}-I_{\widetilde{H}}}\right]^{-1}-I_{\widetilde{H}}\right\|_{\mathcal{L}(\widetilde{H})}
\to 0, \quad t\to \infty,\nonumber\\
\label{eqdefcovosd}
\end{eqnarray}
\noindent with, as before,  $I_{\widetilde{H}}$ denoting the identity operator
on $\widetilde{H}.$   Here, for each $\theta \in \Theta,$
  $\widetilde{\mathcal{M}}_{t,\mathcal{F},\mathcal{A}_{\theta
}}$
   admits the representation
\begin{eqnarray}
\widetilde{\mathcal{M}}_{t,\mathcal{F},\mathcal{A}_{\theta
}}&\underset{\mathcal{L}(\widetilde{H})}{=}&\int_{\Lambda }2\Gamma (1-\alpha (\lambda ,\theta ) )\sin\left((\pi /2)\alpha (\lambda ,\theta ) \right)M_{1/t,\mathcal{F}}(\lambda )
dE_{\lambda }\nonumber\\
&= &\int_{\Lambda }\widetilde{M}_{t,\mathcal{F},\mathcal{A}_{\theta
}}(\lambda )dE_{\lambda },\label{chsd2}
\end{eqnarray}
\noindent where symbols $\alpha (\lambda ,\theta )$ and $M_{\omega ,\mathcal{F}}(\lambda )$ satisfy \textbf{Assumptions III-IV}.  Thus,  $\{X_{t},\ t\in \mathbb{Z}\}$ displays LRD.
\end{proposition}

\begin{proof}
For each $\lambda \in \Lambda,$ under  \textbf{Assumption  IV},     from Theorem 6.5 in \cite{Beran17},  as $t\to \infty,$
\begin{equation}
\left|\left[\int_{-\pi}^{\pi}\exp(it\omega )\frac{M_{\omega ,\mathcal{F}}(\lambda )}{|\omega
|^{\alpha (\lambda ,\theta )}}d\omega \right]\left[\widetilde{M}_{t,\mathcal{F},\mathcal{A}_{\theta
}}(\lambda )t^{\alpha (\lambda ,\theta )-1}\right]^{-1}-1\right|
 \to 0 .\label{lsdosymb}
\end{equation}

Under \textbf{Assumption IV}, for each $\theta \in \Theta ,$ from  equation   (\ref{chsd2}),  the sequence $$\left\{\left|\left[\int_{-\pi}^{\pi}\exp(in\omega )\frac{M_{\omega ,\mathcal{F}}(\lambda )}{|\omega
|^{\alpha (\lambda ,\theta )}}d\omega \right]\left[\widetilde{M}_{n,\mathcal{F},\mathcal{A}_{\theta
}}(\lambda )t^{\alpha (\lambda ,\theta )-1}\right]^{-1}-1\right|,\ n\in \mathbb{N}\right\}$$
\noindent is uniformly bounded in $\lambda \in \Lambda .$ Thus, we can apply
Bounded Convergence Theorem to obtain, from the pointwise convergence (\ref{lsdosymb}),
\begin{eqnarray}&&
\lim_{t\to \infty}\int_{\Lambda }\left|\left[\int_{-\pi}^{\pi}\exp(it\omega )\frac{M_{\omega ,\mathcal{F}}(\lambda )}{|\omega
|^{\alpha (\lambda ,\theta )}}d\omega \right]\left[\widetilde{M}_{t,\mathcal{F},\mathcal{A}_{\theta
}}(\lambda )t^{\alpha (\lambda ,\theta )-1}\right]^{-1}-1\right|dE_{\lambda }\nonumber\\
&&\hspace*{-1cm}\underset{\mathcal{L}(\widetilde{H})}{=}\int_{\Lambda }\lim_{t\to \infty}\left|\left[\int_{-\pi}^{\pi}\exp(it\omega )\frac{M_{\omega ,\mathcal{F}}(\lambda )}{|\omega
|^{\alpha (\lambda ,\theta )}}d\omega \right]\left[\widetilde{M}_{t,\mathcal{F},\mathcal{A}_{\theta
}}(\lambda )t^{\alpha (\lambda ,\theta )-1}\right]^{-1}-1\right|dE_{\lambda }=0.\nonumber\\
\label{srco3}
\end{eqnarray}
From Remark \ref{rem1b}, under \textbf{Assumption I}, $\mathcal{R}_{t}$ is bounded, for every $t\in \mathbb{Z}.$ From  (\ref{srco3}), under \textbf{Assumptions  II-III},  keeping in mind (\ref{sdo2}), we obtain

\begin{eqnarray}
&&\lim_{t\to \infty}\int_{\Lambda }\widetilde{M}_{t,\mathcal{F},\mathcal{A}_{\theta
}}(\lambda )t^{\alpha (\lambda ,\theta )-1}dE_{\lambda }\nonumber\\&&=\lim_{t\to \infty}
\int_{\Lambda }\left[\int_{-\pi}^{\pi}\exp(it\omega )\frac{M_{\omega ,\mathcal{F}}(\lambda )}{|\omega
|^{\alpha (\lambda ,\theta )}}d\omega \right]dE_{\lambda }\nonumber\\
&&=\lim_{t\to \infty}\int_{-\pi}^{\pi}\exp(it\omega )\mathcal{F}_{\omega }d\omega
\nonumber\\
&&=\lim_{t\to \infty}\mathcal{R}_{t,\theta },
\label{ldct}
\end{eqnarray}
\noindent   in the bounded operator norm.

Thus,    equation (\ref{eqdefcovosd})  holds. Therefore, for $M>0,$ sufficiently large,
\begin{eqnarray}&&\hspace*{-1cm}\sum_{t\in \mathbb{Z}}\left\|
\mathcal{R}_{t,\theta} \right\|_{L^{1}(H)}
\geq \sum_{|t|>M}\left\|
\mathcal{R}_{t,\theta}
\right\|_{\mathcal{L}(H)}
\nonumber\\
&& \geq \sum_{|t|>M}\left\|\int_{\Lambda }\widetilde{M}_{t,\mathcal{F},\mathcal{A}_{\theta
}}(\lambda )dE_{\lambda }\right\|_{\mathcal{L}(H)} |t|^{l_{\alpha } (\theta )-1}\nless \infty,
\label{LRDfs}
\end{eqnarray}
\noindent as we wanted to prove.
\end{proof}

\subsection{Examples}\label{ccl}
Some special cases of the LRD family of functional sequences introduced in the spectral domain under \textbf{Assumptions I--IV}
are now analyzed.

\medskip

\subsection{Example 1. Multifractionally integrated functional autoregressive moving averages processes}
\label{SFIFLTS} We consider here, in the stationary case, an extended
family  (see
Remark 9 in \cite{LiRobinsonShang19})  of fractionally integrated functional autoregressive moving
averages models of variable order.

  Let $B$ be a difference operator such that
\begin{equation}
E\|B^{j}X_{t}-X_{t-j}\|_{H}^{2}=0,\quad \forall t,j\in \mathbb{Z}.
\label{ihnorndo}
\end{equation}

Consider the state equation
\begin{eqnarray}&&(1-B)^{\mathcal{A}_{\theta }/2}\Phi_{p}(B)
X_{t}\underset{\mathcal{L}^{2}_{H}(\Omega , \mathcal{A}, P)}{=} \Psi_{q}(B)\eta_{t},\quad \forall t\in \mathbb{Z},\label{farima}
\end{eqnarray}
\noindent where equality holds in the norm of the space $\mathcal{L}^{2}_{H}(\Omega , \mathcal{A}, P).$
Here, $\{\eta_{t},\ t\in \mathbb{Z}\}$ is a
sequence of independent and identically distributed random curves
such that $E[\eta_{t}]=0,$ and $E[\eta_{t}\otimes
\eta_{s}]=\delta_{t,s}\mathcal{R}^{\eta}_{0},$ with
$\mathcal{R}^{\eta}_{0}\in L^{1}(H),$ and $\delta_{t,s}=0,$
for $t\neq s,$ and $\delta_{t,s}=1,$ for $t=s.$  In particular,
\begin{equation}\left\|\mathcal{R}^{\eta}_{0}(h)-\sum_{l\geq 1}^{\infty}
\lambda_{l}(\mathcal{R}^{\eta}_{0})\left\langle \phi_{l},h\right\rangle_{H} \phi_{l}\right\|_{H}=0,\quad
\forall h\in H,\label{rkhs}
\end{equation}
\noindent where $\left\{\phi_{n},\ n\geq 1\right\}$ is an orthonormal  basis of eigenvectors    in $H,$ associated with the eigenvalues
$\{\lambda_{n}(\mathcal{R}^{\eta}_{0}),\ n\geq 1\}.$  Here,
\begin{eqnarray}
\Phi_{p}(B)=1-\sum_{j=1}^{p}\varphi_{j}B^{j},\quad \Psi_{q}(B)=\sum_{j=1}^{q}\psi_{j}B^{j},\nonumber\\
\label{rf}
\end{eqnarray}
\noindent where operators $\varphi_{j},$ $j=1,\dots,p,$ and $\psi_{j},$ $j=1,\dots,q,$ are assumed to be positive self-adjoint bounded operators on
$H,$  admitting the following diagonal spectral decompositions:
\begin{eqnarray}
\varphi_{j}&=&\sum_{l\geq 1}\lambda_{l}(\varphi_{j})\phi_{l}\otimes \phi_{l},\quad j=1,\dots,p\nonumber\\
\psi_{j}&=&\sum_{l\geq 1}\lambda_{l}(\psi_{j})\phi_{l}\otimes \phi_{l},\quad j=1,\dots,q.
\label{de}
\end{eqnarray}
\noindent Also, for each $l\geq 1,$  $\Phi_{p,l}(z)=1-\sum_{j=1}^{p}\lambda_{l}(\varphi_{j})z^{j}$ and $\Psi_{q,l}=\sum_{j=1}^{q}\lambda_{l}(\psi_{j})z^{j}$ have not common
roots, and  their roots are outside of the unit
circle (see also Corollary 6.17 in \cite{Beran17}).

We also assume that, for each $\theta \in \Theta ,$ operator
$\mathcal{A}_{\theta }$ admits the diagonal spectral representation:
\begin{equation}
\mathcal{A}_{\theta }=\sum_{l\geq 1}\alpha (l,\theta
)\phi_{l}\otimes \phi_{l},
\label{lmor}
\end {equation}
\noindent and
\begin{equation}
\int_{-\pi}^{\pi}\sup_{l\geq 1}\frac{\lambda_{l}(\mathcal{R}^{\eta}_{0})}{2\pi}\left|\frac{\Psi_{q,l}(\exp(-i\omega ))}{\Phi_{p,l}(\exp(-i\omega ))}\right|^{2}\left|1-\exp(-i\omega )\right|^{-\alpha (l,\theta
)}d\omega <\infty.\label{a1}
\end{equation}

\noindent Thus, \textbf{Assumption I} holds. Note that, for each $l\geq 1$  and $\theta \in \Theta ,$
\begin{eqnarray}&&(1-\exp(-i\omega  ))^{-\alpha (l,\theta )/2}=\sum_{j=0}^{\infty}a_{j}(l)\exp(-ij\omega )
\nonumber\\
&& a_{j}(l)=\frac{\Gamma (j+\alpha (l,\theta)/2)}{\Gamma (j+1)\Gamma (\alpha (l,\theta)/2)},\quad  j\geq 0.\nonumber\end{eqnarray}
\noindent
Assume that, for each $l\geq 1$ and $\theta \in \Theta ,$
\begin{eqnarray}
&&\sum_{j=0}^{\infty}b_{j,\theta }(l)z^{j}=
 (1-\exp(-iz ))^{-\alpha (l,\theta )/2}\frac{\Psi_{q,l}(z)}{\Phi_{p,l}(z)},\ z\in \mathbb{C}.\label{ifp}\end{eqnarray}

From equations (\ref{farima})--(\ref{ifp}), applying
Corollary 6.17  in \cite{Beran17},
\begin{equation}
X_{t}(\phi_{l})\underset{\mathcal{L}^{2}(\Omega,\mathcal{A},P)}{=}\left(\sum_{j=0}^{\infty}b_{j}(l)B^{j}\right)\eta_{t}(\phi_{l}),\quad l\geq 1,\end{equation}
\noindent and from  Corollary 6.18  in \cite{Beran17}, for each $l\geq 1$  and $\theta \in \Theta ,$  there exists  $\widehat{f}(\omega ,l ,\theta )$ such that
\begin{eqnarray}
\widehat{f}(\omega , l ,\theta )&=& \frac{\lambda_{l}(\mathcal{R}^{\eta}_{0})}{2\pi}\left|\sum_{j=0}^{\infty}b_{j}(l)\exp(-ij\omega )\right|^{2}
\nonumber\\
&=&\frac{\lambda_{l}(\mathcal{R}^{\eta}_{0})}{2\pi}\left|\frac{\Psi_{q,l}(\exp(-i\omega ))}{\Phi_{p,l}(\exp(-i\omega ))}\right|^{2}\left|1-\exp\left(-i\omega \right)\right|^{-\alpha (l,\theta )}.
\label{ifp2}\end{eqnarray}

Thus, for each $l\geq 1$  and $\theta \in \Theta ,$
\begin{equation}
\left\langle \mathcal{R}_{t, \theta }(\phi_{l}),\phi_{l}\right\rangle_{\widetilde{H}}=\int_{-\pi}^{\pi}\exp\left(i\omega t\right)\widehat{f}(\omega , l ,\theta )d\omega, \label{iftcop}
\end{equation}
\noindent and under \textbf{Assumption I} (see equation (\ref{a1})),   we obtain, for each $\theta \in \Theta ,$
\begin{eqnarray}&&\mathcal{R}_{t,\theta }\underset{\mathcal{L}(H)}{=}\int_{-\pi}^{\pi}\exp\left(i\omega t\right)\left[\sum_{l\geq 1}\frac{\lambda_{l}(\mathcal{R}^{\eta}_{0})}{2\pi}\left|\frac{\Psi_{q,l}(\exp(-i\omega ))}{\Phi_{p,l}(\exp(-i\omega ))}\right|^{2}
\right.\nonumber\\
&&\hspace*{2cm}\left.\times \left|1-\exp\left(-i\omega \right)\right|^{-\alpha (l,\theta )}\phi_{l}\otimes \phi_{l}\right]d\omega ,\nonumber
\end{eqnarray}
\noindent under the condition (see equation  (\ref{eqfcpr12}) in Lemma \ref{lempo}),
$$\sigma_{X,\theta}^{2}=\sum_{l\geq 1}\left|\int_{-\pi}^{\pi}\frac{\lambda_{l}(\mathcal{R}^{\eta}_{0})}{2\pi}\left|\frac{\Psi_{q,l}(\exp(-i\omega ))}{\Phi_{p,l}(\exp(-i\omega ))}\right|^{2}\left|1-\exp\left(-i\omega \right)\right|^{-\alpha (l,\theta )}d\omega\right| <\infty .$$

Note that, in this example,   our operator measure satisfies
$$dE(l)(\psi)(\varphi)=d\left\langle E_{\lambda }(\psi),
\varphi\right\rangle =
\phi_{l}(\psi)\phi_{l}(\varphi),\quad   \forall \psi, \varphi,\   l\geq 1.$$
\noindent  Thus, we are considering  a discrete (or point) operator measure, defined from
the common system of eigenvectors $\{\phi_{l},\ l\geq 1\}.$  Equivalently, the spectral family
$\{E_{\lambda_{l}}, \ l\geq 1\}$ admits a representation in terms of a spectral kernel $\widetilde{\Phi},$  defined from the  eigenvectors
$\{\phi_{l},\ l\geq 1\},$  and a point spectral measure (see, e.g., Section 8.2.1 in  \cite{Ramm05}):

\begin{equation}
E_{\lambda_{l}}= \sum_{k=1}^{l}\phi_{k}\otimes \phi_{k}=\sum_{k=1}^{l}\widetilde{\Phi}_{k} ,\quad l\geq 1.\label{eqproy}
\end{equation}

\medskip

    Note that,  since  $\sin(\omega )\sim \omega,$
$\omega \to 0,$
\begin{equation}
\left|1-\exp\left(-i\omega \right)\right|^{-\mathcal{A}_{\theta
}}=[4\sin^{2}(\omega /2)]^{-\mathcal{A}_{\theta }/2}\sim |\omega
|^{-\mathcal{A}_{\theta }},\quad \omega \to 0,\label{ffsvint}
\end{equation}
\noindent where the frequency varying operator
  $\left|1-\exp\left(-i\omega
\right)\right|^{-\mathcal{A}_{\theta }/2}$ is interpreted
as  in
\cite{Charac14}; \cite{duker18};  \cite{Rackauskas};
 \cite{Rackauskasv2}.

Keeping in mind
(\ref{ffsvint}), the following identifications are obtained in
relation to (\ref{sosd})--(\ref{sfasdo}),
\begin{eqnarray}&&
M_{\omega ,\mathcal{F}}(\lambda )= M_{\omega,\mathcal{F}
}(l)=\frac{\lambda_{l}(\mathcal{R}^{\eta}_{0})}{2\pi}
\left|\frac{\Psi_{q,l}\left(\exp(-i\omega
)\right)}{\Phi_{p,l}\left(\exp(-i\omega )\right)}\right|^{2},\quad
l\geq 1,\ \omega \in [-\pi,\pi],
\nonumber\\
&&\left|1-\exp(-i\omega )\right|^{-\alpha (l,\theta )}\sim |\omega
|^{-\alpha (l,\theta )},\quad \omega \to 0,\quad l\geq
1.\label{lrdzf}
\end{eqnarray}
 \noindent Hence, as
$\omega \to 0,$ for each $l\geq 1,$ \begin{equation}\widehat{f}(\omega
,l,\theta )\sim \mathcal{K}_{l} |\omega |^{-\alpha (l,\theta
)},\quad \mathcal{K}_{l}=\frac{\lambda_{l}(\mathcal{R}^{\eta}_{0})}{2\pi}\left|\frac{\Psi_{q,l}(1)}{\Phi_{p,l}(1)}\right|^{2},\quad
\mathcal{K}=\sup_{l\geq 1}\mathcal{K}_{l}<\infty.\label{eqbzf}
\end{equation}

\noindent Assume that  $\Psi_{q,l}$ and $\Phi_{p,l},$ $l\geq 1,$ are such that $$M_{\omega,\mathcal{F}
}(l)=\frac{\lambda_{l}(\mathcal{R}^{\eta}_{0})}{2\pi}
\left|\frac{\Psi_{q,l}\left(\exp(-i\omega
)\right)}{\Phi_{p,l}\left(\exp(-i\omega )\right)}\right|^{2}$$ \noindent  satisfies the conditions  given in \textbf{Assumption IV(ii)}.
Hence, from Proposition \ref{prlrd},  the extended class of fractionally integrated functional
autoregressive moving averages models analyzed here displays LRD (see  also
 Remark 9 in \cite{LiRobinsonShang19}).  Indeed,  $\mathcal{A}_{\theta }/2$ defines the multifractional order of integration.

\subsection{Example 2. Discrete sampling of multifractional  $H$--valued processes in continuous time}
 \label{CSCS}
  Let $H=L^{2}(\mathbb{R},\mathbb{R}),$ and
 $\widetilde{H}=L^{2}(\mathbb{R},\mathbb{C}).$
  Consider  \begin{eqnarray} &&
 d\left\langle E_{\lambda
}(\varphi),\psi\right\rangle_{\widetilde{H}}=\int_{\mathbb{R}}
\widehat{\varphi }(\lambda )\widehat{\psi }(\lambda
)d\lambda\label{sfri}\\
&& \widehat{\psi}(\lambda
)=\int_{\mathbb{R}}\exp\left(-i\left\langle \lambda ,z\right\rangle
\right)\psi(z)dz,\quad \psi \in L^{1}(\mathbb{R})\nonumber\\
&& \widehat{\varphi }(\lambda
)=\int_{\mathbb{R}}\exp\left(-i\left\langle \lambda ,z\right\rangle
\right)\varphi(z)dz,\quad \varphi \in L^{1}(\mathbb{R})\label{sfri22}.
\end{eqnarray}
With  this particular choice,  for $(\lambda ,  \omega)\in
 \mathbb{R}^{2},$ assume that the symbol $f(\omega ,\lambda ,\theta )$ of the spectral density operator $\mathcal{F}_{\omega },$
 with respect to the spectral family $\{E_{\lambda
} ,\ \lambda \in \mathbb{R} \}$ introduced in  (\ref{sfri})--(\ref{sfri22}) is defined as follows:

\begin{eqnarray} &&
f(\omega ,\lambda ,\theta )= |\omega |^{-\alpha (\lambda ,\theta)}
 N_{\omega }(\lambda
) h(\omega ),\label{fexample}\end{eqnarray}
\noindent where $\alpha (\lambda, \theta )$ satisfies \textbf{Assumption IV(i)}, and   $h$ is a positive even taper function of
bounded variation, with  bounded support is the interval $[-\pi,\pi],$
with $h(-\pi)=h(\pi)=0$ (see, e.g.,
\cite{Guyon95}). We also assume that $h$ is
Lipschitz--continuous function, and  $N_{\omega }$   is such that $M_{\omega ,\mathcal{F}}(\lambda )=N_{\omega }(\lambda
) h(\omega )$ satisfies \textbf{Assumption IV(ii)}. Furthermore, for $\omega \in [-\pi,\pi]\backslash \{0\},$
\begin{equation}
\sup_{\lambda \in \mathbb{R}}|f(\omega ,\lambda ,\theta )|<\infty.
 \label{bonfc}
 \end{equation}

 As special case of  (\ref{fexample}), we can consider the tapered  continuous  version of  Example 1 in Section \ref{SFIFLTS}
   $$f(\omega ,\lambda
, \theta)= |\omega |^{-\alpha (\lambda ,\theta)} \frac{P(\lambda
,\omega)}{Q(\lambda ,\omega)}h(\omega)1_{[-\pi,\pi]}(\omega
),\quad  (\lambda ,\omega )\in \mathbb{R}^{2},$$ \noindent
 where the  taper function satisfies the above required conditions, and
  $P$
and $Q$ are     positive  polynomials such that   \textbf{Assumption IV}(ii) holds.  Particularly, when discrete sampling  of  the solution to fractional and multifractional pseudodifferential  evolution equations with Gaussian functional  innovations is considered, one  can implement inference tools from this  framework (see, e.g., \cite{AnhLeonenkoa};  \cite{AnhLeonenkob};  \cite{Kelbert05}).

\section{The convergence    to zero  in $\mathcal{S}(\widetilde{H})$  norm of the bias of the integrated periodogram
operator } \label{sec2}

Theorem \ref{pr1} provides the convergence to zero, in the
Hilbert--Schmidt operator norm,  of the integrated bias of the
periodogram operator.
 Note that, in \cite{Cerovecki17}, weak--convergence of the
covariance operator of the fDFT to the spectral density operator,
and the  convergence of their respective traces is proved. The next
result provides convergence  in $\mathcal{S}(\widetilde{H})$ norm  of the
integrated  covariance operator of the fDFT
to the integrated spectral density operator, in the frequency domain, beyond the SRD condition assumed
in \cite{Cerovecki17}.

\begin{theorem}
\label{pr1} Under \textbf{Assumption I},   the following limit
holds:
$$\left\|\int_{-\pi}^{\pi}\left[\mathcal{F}_{\omega}
-\mathcal{F}_{\omega
}^{(T)}\right]d\omega\right\|_{\mathcal{S}(\widetilde{H})}\to 0,\quad  T\to \infty.$$

\end{theorem}
\begin{proof}  Let $\{\psi_{k},\ k\geq 1\}$
be an orthonormal basis of $\widetilde{H}.$ Under  \textbf{Assumption I},
\begin{eqnarray}
&&\hspace*{-0.3cm} \left\|\int_{-\pi}^{\pi}\left[\mathcal{F}_{\omega }-\mathcal{F}_{\omega }^{(T)}\right]d\omega \right\|_{\mathcal{S}(\widetilde{H})}^{2}\nonumber\\
&&=\sum_{k\geq 1}\int_{-\pi}^{\pi}\int_{-\pi}^{\pi}
\left[\mathcal{F}_{\xi }\mathcal{F}_{\omega }-\mathcal{F}_{\xi }
\mathcal{F}_{\omega }^{(T)}-\mathcal{F}_{\xi }^{(T)}
\mathcal{F}_{\omega }+\mathcal{F}_{\xi }^{(T)} \mathcal{F}_{\omega
}^{(T)}\right](\psi_{k})(\psi_{k})
d\omega d\xi.\nonumber\\
\label{eqintlihnorm}
\end{eqnarray}

From Lemma \ref{lempo}  (see equation (\ref{eqfcpr12})),  for  every $k\geq 1,$
$\mathcal{F}_{\omega }(\psi_{k})(\psi_{k})\in L^{1}([-\pi,\pi]).$ Hence,  for each $k\geq 1,$
\begin{equation}
\mathcal{F}_{\omega }^{(T)}
(\psi_{k})(\psi_{k})\to \mathcal{F}_{\omega }(\psi_{k})(\psi_{k}),  \ T\to \infty,\  \omega \in [-\pi,\pi]\backslash \Lambda_{0}.\label{convfk}
\end{equation}

Applying triangle inequality, for every $T\geq 2,$
\begin{eqnarray}
&&
\left|
\left[\mathcal{F}_{\omega }-\mathcal{F}_{\omega }^{(T)}
\right](\psi_{k})(\psi_{k})\right|\leq 2\left\|\mathcal{F}_{\omega
}\right\|_{L^{1}(\widetilde{H})}<\infty,\quad \omega \in [-\pi,\pi]\backslash \Lambda_{0} ,\nonumber\\
\label{ycil1b}
\end{eqnarray}
\noindent  since  from  (\ref{eqfcpr12}), $$\int_{-\pi}^{\pi}\left\|\mathcal{F}_{\omega
}\right\|_{L^{1}(\widetilde{H})}d\omega<\infty.$$

 \noindent   Hence,   from equations (\ref{convfk})--(\ref{ycil1b}),  keeping in mind (\ref{eqfcpr12}), Dominated Convergence Theorem leads to

\begin{eqnarray}
&&\lim_{T\to \infty}\int_{-\pi}^{\pi}\left|
\left[\mathcal{F}_{\omega }-\mathcal{F}_{\omega }^{(T)}
\right](\psi_{k})(\psi_{k})\right|d\omega \nonumber\\
&&=\int_{-\pi}^{\pi}\lim_{T\to \infty}\left|
\left[\mathcal{F}_{\omega }-\mathcal{F}_{\omega }^{(T)}
\right](\psi_{k})(\psi_{k})\right|d\omega=0,\quad k\geq 1.\label{bct2}
\end{eqnarray}

Note also that, from  Lemma \ref{lempo}, $\|\mathcal{F}_{\omega
}\|_{\mathcal{S}(\widetilde{H})}\in L^{2}\left([-\pi,\pi],
\mathbb{C}\right).$  Hence,   for every $k,l \geq 1,$
  $\mathcal{F}_{\omega }(\psi_{k})(\psi_{l})\in
  L^{2}([-\pi,\pi], \mathbb{C}).$  Particularly, from Young's convolution inequality with $p=2,$ for each $k,l\geq 1,$
  \begin{eqnarray}&&
  \int_{-\pi}^{\pi}|\mathcal{F}_{\omega }^{(T)}(\psi_{k})(\psi_{l})|^{2}d\omega \leq   \int_{-\pi}^{\pi}|\mathcal{F}_{\omega }(\psi_{k})(\psi_{l})|^{2}d\omega .\nonumber\\
  \label{ycihsn}
  \end{eqnarray}
  \noindent Hence,  under \textbf{Assumption I},   applying    Cauchy--Schwarz  and     Jensen's inequalities,   and  (\ref{ycihsn}), from Lemma \ref{lempo}  (see equation (\ref{eqfcpr12})),
  we obtain
 \begin{eqnarray}&&\sum_{k\geq 1}\int_{-\pi}^{\pi}\int_{-\pi}^{\pi}\mathcal{F}_{\xi }^{(T)}\mathcal{F}_{\omega }^{(T)}(\psi_{k})(\psi_{k})d\omega d\xi
  \nonumber\\
  &&=\sum_{k\geq 1}\int_{-\pi}^{\pi}\int_{-\pi}^{\pi}\left\langle \mathcal{F}_{\omega }^{(T)}(\psi_{k}),\mathcal{F}_{\xi }^{(T)}(\psi_{k})
  \right\rangle_{\widetilde{H}}d\omega d\xi
  \nonumber\\
    && \leq \sum_{k\geq 1}\left[\int_{-\pi}^{\pi}\left\|\mathcal{F}_{\omega }^{(T)}(\psi_{k})\right\|_{\widetilde{H}}d\omega \right]
    \left[\int_{-\pi}^{\pi}\left\|\mathcal{F}_{\xi }^{(T)}(\psi_{k})\right\|_{\widetilde{H}}d\xi\right]
  \nonumber\\
  &&\leq 4\pi^{2}\sum_{k\geq 1}\sqrt{\left[\int_{-\pi}^{\pi}\mathcal{F}_{\omega }^{(T)}\mathcal{F}_{\omega}^{(T)}(\psi_{k})(\psi_{k})d\omega \right]}
 \sqrt{\left[\int_{-\pi}^{\pi} \mathcal{F}_{\xi }^{(T)}\mathcal{F}_{\xi }^{(T)}(\psi_{k})(\psi_{k})
     d\xi      \right]      }   \nonumber\end{eqnarray}\begin{eqnarray}
  &&            \leq 4\pi^{2}\sqrt{\sum_{k\geq 1} \int_{-\pi}^{\pi}\left\|\mathcal{F}_{\omega }^{(T)}(\psi_{k})\right\|_{\widetilde{H}}^{2}d\omega} \sqrt{\sum_{k\geq 1}   \int_{-\pi}^{\pi}\left\|\mathcal{F}_{\xi }^{(T)}(\psi_{k})\right\|_{\widetilde{H}}^{2}d\xi }\nonumber\\
    &&   =4\pi^{2}\sqrt{\sum_{k,l\geq 1}\int_{-\pi}^{\pi}|\mathcal{F}_{\omega }^{(T)}(\psi_{k})(\psi_{l})|^{2}d\omega}\sqrt{\sum_{k,l\geq 1}\int_{-\pi}^{\pi}|\mathcal{F}_{\xi }^{(T)}(\psi_{k})(\psi_{l})|^{2}d\xi}\nonumber\\
  &&            \leq 4\pi^{2}\sqrt{\sum_{k,l\geq 1}\int_{-\pi}^{\pi}|\mathcal{F}_{\omega }(\psi_{k})(\psi_{l})|^{2}d\omega}\sqrt{\sum_{k,l\geq 1}\int_{-\pi}^{\pi}|\mathcal{F}_{\xi }(\psi_{k})(\psi_{l})|^{2}d\xi}\nonumber\\
   &&  =4\pi^{2}\int_{-\pi}^{\pi}\left\|\mathcal{F}_{\omega}\right\|^{2} _{\mathcal{S}(\widetilde{H})}d\omega <\infty .
    \label{feprii}
 \end{eqnarray}

Following similar
steps to (\ref{feprii}),
\begin{eqnarray}
&&\sum_{k\geq 1}\int_{-\pi}^{\pi}\int_{-\pi}^{\pi}\mathcal{F}_{\xi
}^{(T)}\mathcal{F}_{\omega }(\psi_{k})(\psi_{k})d\omega  d\xi
\nonumber\\
        &&\leq 4\pi^{2} \sum_{k\geq 1}\sqrt{\int_{-\pi}^{\pi}\int_{-\pi}^{\pi}\mathcal{F}_{\xi}\mathcal{F}_{\xi}(\psi_{k})(\psi_{k})
  \mathcal{F}_{\omega }\mathcal{F}_{\omega }(\psi_{k})(\psi_{k})
    d\omega d\xi}\nonumber\\  &&\leq 4\pi^{2}\int_{-\pi}^{\pi}\left\|\mathcal{F}_{\omega}\right\|^{2} _{\mathcal{S}(\widetilde{H})}d\omega <\infty,
 \label{fepriiv2}
 \end{eqnarray}
\noindent as well as
 \begin{eqnarray}
&&\sum_{k\geq 1}\int_{-\pi}^{\pi}\int_{-\pi}^{\pi}\mathcal{F}_{\xi
}\mathcal{F}_{\omega }^{(T)}(\psi_{k})(\psi_{k})d\omega  d\xi
\nonumber\\
 &&\leq 4\pi^{2}\int_{-\pi}^{\pi}\left\|\mathcal{F}_{\omega}\right\|^{2} _{\mathcal{S}(\widetilde{H})}d\omega <\infty.
 \label{fepriiv3}
 \end{eqnarray}

Furthermore,    from  Lemma \ref{lempo}  (see equation (\ref{eqfcpr12})),
 $\mathcal{F}_{\omega }(\psi_{k})(\psi_{k})\in L^{1}([-\pi,\pi]),$  for  every $k\geq 1.$  Thus, we can consider Young's convolution inequality with $p=1$ leading to
$$\int_{-\pi}^{\pi}\mathcal{F}_{\xi }^{(T)}(\psi_{k})(\psi_{k})d\xi \leq \int_{-\pi}^{\pi}\mathcal{F}_{\xi }(\psi_{k})(\psi_{k})d\xi ,\quad k\geq 1.$$
\noindent  Therefore,
\begin{equation}\int_{-\pi}^{\pi}\|\mathcal{F}_{\xi }^{(T)}\|_{\mathcal{L}(\widetilde{H})}d\xi \leq \int_{-\pi}^{\pi}\|\mathcal{F}_{\xi }\|_{\mathcal{L}(\widetilde{H})}d\xi.\label{sn}\end{equation}

From equations  (\ref{feprii})--(\ref{fepriiv3}), we can apply
  Dominated Convergence Theorem, and keeping in mind
    equations  (\ref{bct2})   and (\ref{sn}), we obtain \begin{eqnarray}
&&\hspace*{-0.6cm}\lim_{T\to \infty} \left\|\int_{-\pi}^{\pi}\left[\mathcal{F}_{\omega }-\mathcal{F}_{\omega }^{(T)}\right]d\omega \right\|_{\mathcal{S}(\widetilde{H})}^{2}\nonumber\\
&&\hspace*{-0.6cm}=\sum_{k\geq 1}\lim_{T\to \infty}\int_{-\pi}^{\pi}\int_{-\pi}^{\pi}\left[\mathcal{F}_{\xi }\mathcal{F}_{\omega }-\mathcal{F}_{\xi }\mathcal{F}_{\omega }^{(T)}-\mathcal{F}_{\xi }^{(T)}\mathcal{F}_{\omega }+\mathcal{F}_{\xi }^{(T)}\mathcal{F}_{\omega }^{(T)}\right](\psi_{k})(\psi_{k})d\omega d\xi.\nonumber\\
&&\hspace*{-0.6cm}\leq \sum_{k\geq 1} 2\left[\int_{-\pi}^{\pi}
\|\mathcal{F}_{\xi }\|_{\mathcal{L}(\widetilde{H})}d\xi\right]
\lim_{T\to \infty} \int_{-\pi}^{\pi} \left|\left[\mathcal{F}_{\omega
}-\mathcal{F}_{\omega }^{(T)}\right]
(\psi_{k})(\psi_{k})\right|d\omega= 0. \label{eqintlihnormv2}
\end{eqnarray}

\end{proof}

\section{Semiparametric  estimation in the spectral domain}
\label{sththeslrd} This section introduces the estimation
methodology adopted in the functional spectral domain. Theorem
\ref{theslrd} derives the  weak  consistency   of the formulated
parametric  estimator of the long--memory operator.

Under \textbf{Assumptions I-IV},
 let $\Theta
\subset \mathbb{R}^{p},$ $p\geq 1,$ be a compact subset of
$\mathbb{R}^{p}.$ Assume that the true parameter value $\theta_{0}$
lies in the interior of $\Theta ,$ denoted as $\mbox{int} \ \Theta
.$
 The symbol   $\alpha : \mathbb{R}\times \Theta
 \longrightarrow (0,1)$ is such that
 $\alpha (\cdot ,\theta_{1})\neq \alpha (\cdot ,\theta _{2}),$
 for $\theta_{1}\neq \theta_{2},$ for every
 $\theta_{1},\theta_{2}\in \Theta .$   Thus,  under
 (\ref{sosd}), we get  indentifiability in the semiparametric model.
Denote by $\widehat{\theta }_{T}$   the estimator of the true
parameter value $\theta_{0},$ based on a functional sample of size $T.$  Hence,   $\widehat{\alpha}_{T} (\lambda ,\theta )=
 \alpha (\lambda , \widehat{\theta }_{T})$ provides the
  parametric estimator of the symbol
$\alpha(\lambda ,\theta )$ of $\mathcal{A}_{\theta }.$

Let now introduce the elements involved in the definition of our operator loss function, to compute the minimum contrast estimator $\widehat{\theta }_{T},$  under suitable conditions.
Specifically, for each $\omega \in [-\pi,\pi ],$ the weighting operator $\mathcal{W}_{\omega}$ is introduced as a bounded positive self--adjoint operator admitting the following spectral representation:
\begin{eqnarray}&& \mathcal{W}_{\omega}=
\int_{\Lambda }W(\omega ,\lambda ,\beta ) dE_{\lambda },
\nonumber\\
&&=\int_{\Lambda } \widetilde{W}(\lambda ) |\omega |^{\beta
}dE_{\lambda },\quad  \beta>0.\label{eqdeftildew}
\end{eqnarray}\noindent
 In particular, the symbol $W(\omega ,\lambda ,\beta )$  of operator $\mathcal{W}_{\omega}$ factorizes, in terms of
  $\widetilde{W}(\lambda )$ and  $|\omega |^{\beta
},$ with  $\widetilde{W}$ defining the symbol of a positive self--adjoint  operator
 $\widetilde{\mathcal{W}}\in
\mathcal{L}(\widetilde{H})$  such that
\begin{equation}m_{\widetilde{\mathcal{W}}}=\inf_{\psi\in \widetilde{H};\
   \|\psi\|_{\widetilde{H}}=1}\left\langle \widetilde{\mathcal{W}}(\psi),
   \psi\right\rangle_{\widetilde{H}},\quad M_{\widetilde{\mathcal{W}}}=
   \sup_{\psi\in \widetilde{H};\
    \|\psi\|_{\widetilde{H}}=1}\left\langle \widetilde{\mathcal{W}}(\psi),
    \psi\right\rangle_{\widetilde{H}}.
\label{dbwo}
\end{equation}

For each $\theta
\in \Theta ,$  the normalizing operator $\sigma^{2}_{\theta }$  is computed as follows:

\begin{eqnarray}&&
\sigma^{2}_{\theta }=\int_{-\pi}^{\pi}
\mathcal{F}_{\omega ,\theta }\mathcal{W}_{\omega}d\omega \nonumber\\ & &=\int_{-\pi}^{\pi} \int_{\Lambda }
\frac{M_{\omega ,\mathcal{F}}(\lambda)\widetilde{W}(\lambda )}{ |\omega|^{\alpha
(\lambda ,\theta )- \beta }}dE_{\lambda }d\omega. \label{eqdeftildew2}
\end{eqnarray}

Thus,   the symbol  $\Sigma^{2}_{\theta }$ of $\sigma^{2}_{\theta }$ is given by
\begin{eqnarray}&&\Sigma^{2}_{\theta }(\lambda )=\int_{-\pi}^{\pi} \frac{M_{\omega ,\mathcal{F}}(\lambda)\widetilde{W}(\lambda )}{ |\omega|^{\alpha
(\lambda ,\theta )- \beta }}d\omega ,\quad \forall \lambda \in  \Lambda.
\label{eqdeftildew3}
\end{eqnarray}
Under \textbf{Assumption IV(ii)} (see equation (\ref{dbnpp})),   and (\ref{dbwo}),  for every $\lambda \in \Lambda ,$
\begin{eqnarray}
&&m m_{\widetilde{\mathcal{W}}}
\left(\left[\int_{-\pi}^{-1}+\int_{1}^{\pi}\right]|\omega|^{-L(\theta )+\beta }d\omega +  \int_{-1}^{1}  |\omega|^{-l(\theta )+\beta }d\omega \right)\nonumber\\
&&=m m_{\widetilde{\mathcal{W}}}\left[\frac{(-\pi)^{1+\beta-L(\theta )}-(-1)^{1+\beta-L(\theta )}}{1+\beta-L(\theta )}\right.\nonumber\\
&&\left.+\frac{(\pi)^{1+\beta-L(\theta )}-1}{1+\beta-L(\theta )}+\frac{(-1)^{1-l(\theta )+\beta }}{1-l(\theta )+\beta}+\frac{1}{1-l(\theta )+\beta}\right]
\nonumber\end{eqnarray}\begin{eqnarray}
&&\leq
\Sigma^{2}_{\theta }(\lambda )\leq M M_{\widetilde{\mathcal{W}}}\left(\left[\int_{-\pi}^{-1}+\int_{1}^{\pi}\right]|\omega|^{-l(\theta )+\beta }d\omega +  \int_{-1}^{1}  |\omega|^{-L(\theta )+\beta }d\omega \right)\nonumber\\
&&=M M_{\widetilde{\mathcal{W}}}\left[\frac{(-\pi)^{1+\beta-l(\theta )}-(-1)^{1+\beta-l(\theta )}}{1+\beta-l(\theta )}\right.\nonumber\\
&&\left.+\frac{(\pi)^{1+\beta-l(\theta )}-1}{1+\beta-l(\theta )}+\frac{(-1)^{1-L(\theta )+\beta }}{1-L(\theta )+\beta}+\frac{1}{1-L(\theta )+\beta}\right].\label{ubno2b}
\end{eqnarray}
Thus,  $\sigma^{2}_{\theta }$ is a bounded operator. The symbol of $[\sigma^{2}_{\theta }]^{-1}$ is given by

\begin{eqnarray}
&&\left[\Sigma^{2}_{\theta }(\lambda )\right]^{-1}=
\left[\int_{-\pi}^{\pi}
\frac{M_{\omega ,\mathcal{F}}(\lambda)\widetilde{W}(\lambda )}{ |\omega|^{\alpha
(\lambda ,\theta )- \beta }}d\omega \right]^{-1},\quad \lambda \in \Lambda .\label{ubno2c}
\end{eqnarray}

From (\ref{ubno2b}),  for every $\lambda \in \Lambda ,$
\begin{eqnarray}&&\left\{m m_{\widetilde{\mathcal{W}}}\left[\frac{(-\pi)^{1+\beta-L(\theta )}-(-1)^{1+\beta-L(\theta )}}{1+\beta-L(\theta )}\right.\right.\nonumber\\
&&\left.\left.+\frac{(\pi)^{1+\beta-L(\theta )}-1}{1+\beta-L(\theta )}+\frac{(-1)^{1-l(\theta )+\beta }}{1-l(\theta )+\beta}+\frac{1}{1-l(\theta )+\beta}\right]\right\}^{-1}\nonumber\\
&&\geq \left[\Sigma^{2}_{\theta }(\lambda )\right]^{-1}\geq \left\{M M_{\widetilde{\mathcal{W}}}\left[\frac{(-\pi)^{1+\beta-l(\theta )}-(-1)^{1+\beta-l(\theta )}}{1+\beta-l(\theta )}\right.\right.\nonumber\\
&&\left.\left.+\frac{(\pi)^{1+\beta-l(\theta )}-1}{1+\beta-l(\theta )}+\frac{(-1)^{1-L(\theta )+\beta }}{1-L(\theta )+\beta}+\frac{1}{1-L(\theta )+\beta}\right]\right\}^{-1}.
\label{lbno}
\end{eqnarray}

\noindent Hence, $[\sigma^{2}_{\theta }]^{-1}$ is strictly positive and bounded.

From (\ref{eqdeftildew2}), we can consider
 the following  factorization of the spectral density operator,
 for $(\omega ,\theta )\in [-\pi,\pi]\backslash \{0\}\times \Theta ,$
 \begin{equation}\mathcal{F}_{\omega ,\theta }= \sigma^{2}_{\theta }
\Upsilon_{\omega ,\theta }=\Upsilon_{\omega ,\theta }
\sigma^{2}_{\theta },\label{fsdo}
\end{equation} \noindent where,  for each $\theta \in \Theta,$  and $\omega \in [-\pi,\pi],$  $\omega \neq 0,$
\begin{eqnarray}\Upsilon_{\omega ,\theta }
&=&\int_{\Lambda  }\Upsilon (\omega ,\lambda ,\theta )
dE_{\lambda }
\nonumber\\
&=&\int_{\Lambda  }
\frac{M_{\omega ,\mathcal{F} }(\lambda
)}{|\omega |^{\alpha (\lambda ,\theta )}\Sigma^{2}_{\theta }(\lambda )}dE_{\lambda }.
\label{standarisdo}
\end{eqnarray}

 From
equations (\ref{eqdeftildew})--(\ref{standarisdo}),
 for each  $\theta \in \Theta ,$ and any  $\varrho, \psi \in \widetilde{H},$
\begin{eqnarray}&&
\int_{-\pi}^{\pi}\Upsilon_{\omega ,\theta }\mathcal{W}_{\omega }
(\varrho)(\psi )d\omega =\int_{\Lambda }d\left\langle
E_{\lambda }(\varrho),\psi
\right\rangle_{\widetilde{H}}=\left\langle
\varrho,\psi\right\rangle_{\widetilde{H}}.
\label{identityintegralop}
\end{eqnarray} Equivalently,
$\int_{-\pi}^{\pi}\Upsilon_{\omega ,\theta }\mathcal{W}_{\omega
}d\omega $ coincides with the identity operator $I_{\widetilde{H}}$
on $\widetilde{H},$ for each $\theta \in \Theta.$

Let us now consider the empirical operator $\mathbf{U}_{T, \theta}$
given by, for each $\theta \in \Theta ,$ \begin{equation} [\mathbf{U}_{T,
\theta}]=-\int_{-\pi}^{\pi }p_{\omega }^{(T)}
\ln\left(\Upsilon_{\omega ,\theta }\right)\mathcal{W}_{\omega
} d\omega ,
\label{eco}
\end{equation}
\noindent where $T$ denotes as before the sample size. Its
theoretical counterpart $U_{\theta }$ is defined, for each $\theta
\in \Theta ,$ as
\begin{eqnarray} &&  U_{\theta }
=-\int_{-\pi}^{\pi }\mathcal{F}_{\omega ,\theta_{0}}
\ln\left(\Upsilon_{\omega ,\theta }\right)\mathcal{W}_{\omega
}d\omega \nonumber\\
&&=-\int_{-\pi}^{\pi}\int_{\Lambda  } \frac{M_{\omega ,\mathcal{F}
}(\lambda )\widetilde{W}(\lambda )}{|\omega |^{\alpha (\lambda,
\theta_{0} )-\beta }}\ln\left(\Upsilon (\omega , \lambda ,\theta
)\right)dE_{\lambda }d\omega. \label{ecoth}
\end{eqnarray}

\begin{remark}
\label{fus}
Note that, under \textbf{Assumption IV(i)}, for each $\theta \in \Theta ,$  $U_{\theta }\in
 \mathcal{L}(\widetilde{H})$
for any $\beta >0.$  Specifically,   for every $\lambda \in \Lambda ,$
$\frac{M_{\omega ,\mathcal{F}
}(\lambda )}{|\omega |^{\alpha (\lambda,
\theta_{0} ) }}\in L^{1}([-\pi,\pi]),$
and $\ln\left(\Upsilon (\omega , \lambda ,\theta
)\right)W(\omega ,\lambda ,\beta )\in L^{1}([-\pi,\pi]),$  with
\begin{eqnarray}
&&\sup_{\lambda \in \Lambda } \left|-\int_{-\pi}^{\pi}\frac{M_{\omega ,\mathcal{F}
}(\lambda )}{|\omega |^{\alpha (\lambda,
\theta_{0} ) }}\ln\left(\Upsilon (\omega , \lambda ,\theta
)\right)W(\omega ,\lambda ,\beta )d\omega \right|<\infty.
\label{ecoth2}
\end{eqnarray}
 In addition, for $T$ large, $\mathbf{U}_{T, \theta}\in \mathcal{L}(\widetilde{H})$  a.s. (see Theorem \ref{theslrd} below).
\end{remark}

We now consider the loss operator $\mathcal{K}(\theta_{0},
\theta )$ to be minimized, with respect to $\theta ,$ in the
operator norm.  Specifically,  consider, for each $\theta \in \Theta
,$
\begin{eqnarray}&&
[\mathcal{K}(\theta_{0}, \theta )]=\int_{-\pi}^{\pi}\mathcal{F}_{\omega ,\theta_{0}}
\ln\left(\Upsilon_{\omega , \theta_{0} }\Upsilon_{\omega ,\theta
}^{-1}\right)\mathcal{W}_{\omega }d\omega
\nonumber\\
&&=[U_{\theta }-U_{ \theta_{0}}]. \label{eqkld}
\end{eqnarray}
\noindent From Remark \ref{fus}, for each $\theta \in \Theta ,$
$\mathcal{K}(\theta_{0}, \theta )\in \mathcal{L}(\widetilde{H}).$ Furthermore,  the symbol of $\mathcal{K}(\theta_{0}, \theta )$
is given by
\begin{equation}
\int_{-\pi}^{\pi}\frac{M_{\omega ,\mathcal{F}}(\lambda)}{ |\omega|^{\alpha
(\lambda ,\theta_{0} )}}\ln\left(\frac{\Upsilon (\omega , \lambda , \theta_{0})}{\Upsilon (\omega , \lambda , \theta )
}\right) W(\omega, \lambda ,\beta )d\omega , \quad \lambda \in \Lambda , \quad \theta \in \Theta.
\label{ecoth3}
\end{equation}

 Operator
$[\sigma^{2}_{\theta_{0}}]^{-1}\mathcal{K}(\theta_{0}, \theta )$
could be interpreted as a  weighted  Kullback--Leibler divergence
operator, measuring the  discrepancy  between the two semiparametric
functional spectral models $\Upsilon_{\omega ,\theta_{0}}$ and
$\Upsilon_{\omega ,\theta },$  for each $\theta \in \Theta $ (see,
e.g., \cite{Cover91}). Note  that, from equations
(\ref{eqdeftildew2})--(\ref{eqkld}),
  applying Jensen's
   inequality, for every $k\geq 1,$  and $\theta \in \Theta ,$
  \begin{eqnarray}
 && -[\mathcal{K}(\theta_{0},\theta )](\psi_{k})(\psi_{k})
  \nonumber\\
  &&\leq\left\|\sigma_{\theta }^{2}\right\|_{\mathcal{L}(\widetilde{H})} \ln\left(\int_{-\pi}^{\pi}\int_{\Lambda }
  \Upsilon (\omega ,\lambda ,\theta )W(\omega ,\lambda ,\beta )
  d\left\langle E_{\lambda
}(\psi_{k}),\psi_{k}\right\rangle_{\widetilde{H}}d\omega \right)\nonumber\\
&&=\left\|\sigma_{\theta }^{2}\right\|_{\mathcal{L}(\widetilde{H})}
\ln\left(\int_{-\pi}^{\pi}\Upsilon_{\omega ,\theta
}\mathcal{W}_{\omega }(\psi_{k})(\psi_{k})
  d\omega \right)
\nonumber\\&&=\left\|\sigma_{\theta
}^{2}\right\|_{\mathcal{L}(\widetilde{H})}
\ln\left(\|\psi_{k}\|_{\widetilde{H}}^{2}\right)=0. \label{eqjibb}
  \end{eqnarray}\noindent  for any orthonormal basis $\{\psi_{k},\ k\geq 1\}$ of $\widetilde{H}.$
From equation (\ref{eqjibb}), for every $k\geq 1,$
  and $\theta \in \Theta ,$
  $[\mathcal{K}(\theta_{0},\theta )](\psi_{k})(\psi_{k})\geq 0.$ Thus, $\left\{\mathcal{K}(\theta_{0},\theta ),\ \theta \in \Theta \right\}$ is a parametric family of non--negative self--adjoint  bounded operators  such that \begin{eqnarray}
 && \|\mathcal{K}(\theta_{0},\theta )\|_{\mathcal{L}(\widetilde{H})}=\sup_{k\geq 1}[\mathcal{K}(\theta_{0},\theta )](\psi_{k})(\psi_{k})>0,\quad \theta  \neq  \theta_{0}\nonumber\\
  && \|\mathcal{K}(\theta_{0},\theta )\|_{\mathcal{L}(\widetilde{H})}=\sup_{k\geq 1}[\mathcal{K}(\theta_{0},\theta )](\psi_{k})(\psi_{k})=0\ \Leftrightarrow \  \theta =  \theta_{0}.
  \label{eqjibbh}
  \end{eqnarray}

  Hence, from equations (\ref{eqjibbh}),
\begin{eqnarray} \theta_{0}&=&
\mbox{arg} \
 \min_{\theta \in \Theta }
 \left\|[\mathcal{K}(\theta_{0},\theta )]
 \right\|_{\mathcal{L}(\widetilde{H})}\nonumber\\
&=&\mbox{arg} \
 \min_{\theta \in \Theta } \sup_{k\geq 1}\mathcal{K}(\theta_{0},\theta
  )(\psi_{k})(\psi_{k})\nonumber\\
&=&\mbox{arg} \ \min_{\theta \in \Theta } \sup_{k\geq 1}U_{\theta
}(\psi_{k})(\psi_{k}). \label{mfethetabb}\end{eqnarray}

We then consider the following estimator  $\widehat{\theta }_{T}$   computed from the empirical contrast operator $U_{T, \theta}$  in
 (\ref{eco}), and a given orthonormal basis $\{\psi_{k},\
k\geq 1\}$ of $\widetilde{H}:$

\begin{eqnarray}
\widehat{\theta }_{T}&=& \mbox{arg} \ \min_{\theta \in \Theta }
\sup_{k\geq 1}U_{T,\theta }(\psi_{k})(\psi_{k}). \label{mfetheta}
\end{eqnarray}

\begin{theorem}
\label{theslrd} Let $\{X_{t},\ t\in \mathbb{Z}\}$  be a
stationary zero--mean
 Gaussian  functional sequence satisfying  \textbf{Assumptions I--IV}.
  Consider in \textbf{Assumption IV(ii)}  the particular case where $\mathcal{M}_{\omega  ,\mathcal{F}}$  satisfies, for any $\xi>0,$
\begin{equation}\lim_{\omega \to 0}\left\|\mathcal{M}_{\omega /\xi ,\mathcal{F}}\mathcal{M}_{ \omega  ,\mathcal{F}}^{-1}-I_{\widetilde{H}}\right\|_{\mathcal{L}(\widetilde{H})}=0.\label{svfo}
\end{equation}
 Under  the conditions reflected
 in
equations (\ref{eqdeftildew})--(\ref{eqjibb}), for $\beta >1,$ we then have

\begin{eqnarray}
&& E\left\|\int_{-\pi}^{\pi}\left[p_{\omega
}^{(T)}-\mathcal{F}_{\omega ,\theta_{0}}\right]\mathcal{W}_{\omega
,\theta }d\omega \right\|_{\mathcal{S}(\widetilde{H})} \to
0,\quad T\to \infty, \label{eqtlimits1}
\end{eqnarray}
\noindent where, for $(\omega , \theta)\in [-\pi,\pi]\times \Theta ,$
\begin{equation}\mathcal{W}_{\omega
,\theta }= \ln\left(\Upsilon_{\omega ,\theta
}\right)\mathcal{W}_{\omega
}.\label{eqwo}
\end{equation}

Furthermore, the
estimator $\widehat{\theta }_{T}$  in (\ref{mfetheta})
satisfies
$$\widehat{\theta }_{T}\to_{P} \theta_{0},\quad T\to \infty,$$
\noindent where $\to_{P}$ denotes convergence in probability.
\end{theorem}

\begin{proof}

The operator   $\mathcal{W}_{\omega ,\theta }$ introduced in (\ref{eqwo}) admits the spectral representation
  \begin{eqnarray}&&\hspace*{-1cm}\mathcal{W}_{\omega
,\theta }= \int_{\Lambda}\left[\ln\left(M_{\omega }(\lambda
)\right)-\ln\left(\Sigma_{\theta }^{2}(\lambda )
\right)\right.\nonumber\\
&&\left. \hspace*{1.5cm}-\alpha (\lambda ,\theta ) \ln\left(
|\omega |\right)\right] \widetilde{W}(\lambda )|\omega |^{\beta
}d E_{\lambda },\label{defwo}
\end{eqnarray}
 \noindent for
$\omega \in [-\pi,\pi],$ and $\theta \in \Theta.$ From
(\ref{defwo}),
 \begin{eqnarray}&&
 \left\|\mathcal{W}_{\omega ,\theta }\right\|_{\mathcal{L}(\widetilde{H})}\leq \left\|\ln\left(\mathcal{M}_{\omega ,\mathcal{F}}\right)|\omega |^{\beta }\widetilde{\mathcal{W}}_{\omega }\right\|_{\mathcal{L}(\widetilde{H})}\nonumber\\
 &&\hspace*{1cm}+\left\|\mathcal{A}_{\theta }\ln\left( |\omega |\right)
 |\omega |^{\beta }\widetilde{\mathcal{W}}_{\omega }\right\|_{\mathcal{L}(\widetilde{H})}+\left\|\ln \left(\sigma^{2}_{\theta }\right)
 |\omega |^{\beta }\widetilde{\mathcal{W}}_{\omega }\right\|_{\mathcal{L}(\widetilde{H})}\nonumber\\
 &&\leq \ln(M)\pi^{\beta }M_{\widetilde{\mathcal{W}}}+
 L(\theta )\ln(\pi)\pi^{\beta }M_{\widetilde{\mathcal{W}}}+\left\|\ln \left(\sigma^{2}_{\theta }\right)\right\|_{\mathcal{L}(\widetilde{H})}
 \pi^{\beta }M_{\widetilde{\mathcal{W}}},
\label{bonwo}\end{eqnarray}\noindent
 for every
$\theta \in \Theta ,$ $\omega \in [-\pi,\pi],$ and $\beta >0.$

From (\ref{bonwo}),
\begin{eqnarray}&&\sup_{\omega \in [-\pi,\pi]}\left\|\mathcal{W}_{\omega ,\theta }\right\|_{\mathcal{L}(\widetilde{H})}\leq
\ln(M)\pi^{\beta }M_{\widetilde{\mathcal{W}}}+
 L(\theta )\ln(\pi)\pi^{\beta }M_{\widetilde{\mathcal{W}}}\nonumber\\
 &&\hspace*{2cm} +\left\|\ln \left(\sigma^{2}_{\theta }\right)\right\|_{\mathcal{L}(\widetilde{H})}
 \pi^{\beta }M_{\widetilde{\mathcal{W}}}=\mathcal{H}(\theta ).\nonumber\\
\label{bonwovv}\end{eqnarray}

Thus, the
family $\left\{ \mathcal{W}_{\omega ,\theta },\ \omega \in
[-\pi,\pi]\right\}$ is equicontinuous, for any $\theta \in \Theta.$
We first  prove that the following limits hold, for each  $\theta
\in \Theta ,$
\begin{eqnarray}
&&\left\|\int_{-\pi}^{\pi}\left[\mathcal{F}_{\omega, \theta_{0} }^{(T)}
-\mathcal{F}_{\omega ,\theta_{0}}\right] \mathcal{W}_{\omega ,\theta }
d\omega \right\|_{\mathcal{S}(\widetilde{H})}\to 0,\quad T\to \infty
\label{eqtlimits0}\\
&& E\left\|\int_{-\pi}^{\pi}\left[p_{\omega
}^{(T)}-\mathcal{F}^{(T)}_{\omega,\theta_{0}}\right]\mathcal{W}_{\omega
,\theta }d\omega \right\|^{2}_{\mathcal{S}(\widetilde{H})} \to
0,\quad T\to \infty, \label{eqtlimits1b}
\end{eqnarray}\noindent
where $E\left(p_{\omega }^{(T)}\right)=
\mathcal{F}^{(T)}_{\omega,\theta_{0}}.$

 From Theorem \ref{pr1},
 and equation (\ref{bonwovv}),
\begin{eqnarray}
&&\left\|\int_{-\pi}^{\pi}\left[E\left(p_{\omega }^{(T)}
\right)-\mathcal{F}_{\omega ,\theta_{0}}\right]
\mathcal{W}_{\omega ,\theta }d\omega\right\|_{\mathcal{S}(\widetilde{H})}\nonumber\\
&& \leq\mathcal{H}(\theta )\left\| \int_{-\pi}^{\pi}
\left[\mathcal{F}^{(T)}_{\omega,\theta_{0}}
-\mathcal{F}_{\omega,\theta_{0}}\right]d\omega
\right\|_{\mathcal{S}(\widetilde{H})}\to
0,\quad T\to \infty.
 \label{limuk}
\end{eqnarray}

Under the Gaussian  distribution of
$\{ X_{t},\ t\in \mathbb{Z}\},$    applying
Fourier Transform Inversion Formula,  we obtain
\begin{eqnarray}&&
E\left\| \int_{-\pi}^{\pi}\left[p_{\omega
}^{(T)}-\mathcal{F}^{(T)}_{\omega,\theta_{0}}\right]
\mathcal{W}_{\omega ,\theta }d\omega\right\|_{\mathcal{S}
(\widetilde{H})}^{2}\nonumber\\
&& =\sum_{k\geq 1}\int_{-\pi}^{\pi}\int_{-\pi}^{\pi}\left[E\left[p_{\xi
}^{(T)}p_{\omega }^{(T)}\right]+ \mathcal{F}^{(T)}_{\xi,\theta_{0}}
\mathcal{F}^{(T)}_{\omega,\theta_{0}}
-\mathcal{F}^{(T)}_{\xi,\theta_{0}} E\left[p_{\omega
}^{(T)}\right]\right.\nonumber\\
&&\left. \hspace*{0.5cm}-E\left[p_{\xi
}^{(T)}\right]\mathcal{F}^{(T)}_{\omega,\theta_{0}}\right]
\mathcal{W}^{\star}_{\xi ,\theta }\mathcal{W}_{\omega ,\theta }
(\psi_{k})(\psi_{k})d\omega d\xi \nonumber\\
&& =\sum_{k\geq 1}\int_{-\pi}^{\pi}\int_{-\pi}^{\pi}\left[E\left[p_{\xi
}^{(T)}p_{\omega
}^{(T)}\right]-\mathcal{F}^{(T)}_{\xi,\theta_{0}}\mathcal{F}^{(T)}_{\omega,\theta_{0}}\right]
\mathcal{W}^{\star}_{\xi ,\theta
}\mathcal{W}_{\omega ,\theta }
 (\psi_{k})(\psi_{k})d\omega d\xi\nonumber\\
&& =\frac{1}{(2\pi T)^{2}}\sum_{k\geq 1}\int_{-\pi}^{\pi}\int_{-\pi}^{\pi}
\left[\sum_{t_{1}, s_{1},t_{2},s_{2}=1}^{T}\exp\left(-i\omega
(t_{1}-s_{1})-i\xi (t_{2}-s_{2})\right)\right. \nonumber\\
&&\left.\times \left[E\left[X_{t_{1}}\otimes X_{s_{1}}\otimes
X_{t_{2}}\otimes X_{s_{2}}\right]-E\left[X_{t_{1}}\otimes
X_{s_{1}}\right]E\left[X_{t_{2}}\otimes X_{s_{2}}\right]
\right]\right]\nonumber\\
&& \hspace*{3cm} \mathcal{W}^{\star}_{\xi ,\theta
}\mathcal{W}_{\omega ,\theta } (\psi_{k})(\psi_{k})d\omega
d\xi\nonumber\\ && =\frac{1}{(2\pi T)^{2}}\sum_{k\geq
1}\int_{-\pi}^{\pi}\int_{-\pi}^{\pi}\left[\sum_{t_{1},
s_{1},t_{2},s_{2}=1}^{T}\exp\left(-i\omega
(t_{1}-s_{1})-i\xi (t_{2}-s_{2})\right)\right. \nonumber\\
&&\left.\times \left[E\left[X_{t_{1}}\otimes
X_{t_{2}}\right]E\left[X_{s_{1}}\otimes X_{s_{2}}\right]
 +E\left[X_{t_{1}}\otimes X_{s_{2}}\right]E\left[X_{t_{2}}\otimes
X_{s_{1}}\right] \right]\right]\nonumber\\&&
\hspace*{3cm}\mathcal{W}^{\star}_{\xi ,\theta }\mathcal{W}_{\omega
,\theta } (\psi_{k})(\psi_{k})d\omega d\xi \nonumber\\ &&
=\frac{2\pi}{T}\sum_{k\geq
1}\int_{-\pi}^{\pi}\int_{-\pi}^{\pi}\int_{-\pi}^{\pi}\int_{-\pi}^{\pi}\mathcal{F}_{\widetilde{\omega
},\theta_{0}}\mathcal{F}_{\widetilde{\xi},\theta_{0}}\left[\frac{1}{[2\pi]^{3}T}\sum_{t_{1},
s_{1},t_{2},s_{2}=1}^{T}\exp\left(i t_{1}(\widetilde{\omega }-\omega
)\right)\right.\nonumber\\
&&\left.\hspace*{3cm}\times \exp\left(i s_{1}(\omega +\widetilde{\xi})+it_{2}(-\xi-\widetilde{\omega
})+is_{2}(\xi-\widetilde{\xi })\right)\right.\nonumber\\&&\left.
+\exp\left(it_{1}(\widetilde{\omega }-\omega)+is_{1}(\omega
+\widetilde{\xi })+it_{2}(-\xi-\widetilde{\xi }
)+is_{2}(\xi-\widetilde{\omega
})\right)\right]\nonumber\\
&& \hspace*{2cm}\mathcal{W}^{\star}_{\xi ,\theta
}\mathcal{W}_{\omega ,\theta } (\psi_{k})(\psi_{k})d\omega d\xi
d\widetilde{\omega }
d\widetilde{\xi }\nonumber\\
&&=\frac{2\pi}{T}\sum_{k\geq
1}\int_{-\pi}^{\pi}\int_{f_{1}(\omega )}^{f_{2}(\omega )}\int_{g_{1}(\omega )}^{g_{2}(\omega)}\int_{h_{1}(\omega, u_{1})}^{h_{2}(\omega ,u_{1})}\Phi_{T}^{4}(u_{1},u_{2},u_{3})
\mathcal{F}_{u_{1}+\omega ,\theta_{0}}\mathcal{F}_{u_{2}-\omega
,\theta_{0}}\nonumber\\
&&\hspace*{3cm}\mathcal{W}^{\star}_{-(u_{1}+u_{3}+\omega ),\theta
}\mathcal{W}_{\omega ,\theta } (\psi_{k})(\psi_{k})du_{3}du_{2}du_{1}d\omega
\nonumber\\
&&+\frac{2\pi}{T}\sum_{k\geq
1}\int_{-\pi}^{\pi}\int_{f_{1}(\omega )}^{f_{2}(\omega )}\int_{g_{1}(\omega )}^{g_{2}(\omega)}\int_{\widetilde{h}_{1}(\omega, \widetilde{u}_{1})}^{\widetilde{h}_{2}(\omega ,\widetilde{u}_{1})}\Phi_{T}^{4}(\widetilde{u}_{1},\widetilde{u}_{2},\widetilde{u}_{3})
\mathcal{F}_{\widetilde{u}_{1}+\omega
,\theta_{0}}\mathcal{F}_{\widetilde{u}_{2}-\omega
,\theta_{0}}
\nonumber\end{eqnarray}\begin{eqnarray}
&&\hspace*{3cm}\mathcal{W}^{\star}_{\widetilde{u}_{3}-\widetilde{u}_{1}
-\omega ,\theta } \mathcal{W}_{\omega ,\theta }
(\psi_{k})(\psi_{k})d\widetilde{u}_{3}d\widetilde{u}_{2}d\widetilde{u}_{1}d\omega
\nonumber\\
&& = \frac{2\pi}{T}\int_{-\pi}^{\pi}\int_{f_{1}(\omega )}^{f_{2}(\omega )}\int_{g_{1}(\omega )}^{g_{2}(\omega)}\int_{h_{1}(\omega, u_{1})}^{h_{2}(\omega ,u_{1})}
\Phi_{T}^{4}(u_{1},u_{2},u_{3})\nonumber\\
&& \hspace*{2cm}\times
\left\langle \mathcal{F}_{u_{1}+\omega
,\theta_{0}}\mathcal{W}_{\omega,\theta }, \mathcal{F}_{u_{2}-\omega
,\theta_{0}}\mathcal{W}_{-(u_{1}+u_{3}+\omega ),\theta }
\right\rangle_{\mathcal{S}(\widetilde{H})}du_{3} du_{2}du_{1}d\omega \nonumber\\
&&+\frac{2\pi}{T}\int_{-\pi}^{\pi}\int_{f_{1}(\omega )}^{f_{2}(\omega )}\int_{g_{1}(\omega )}^{g_{2}(\omega)}\int_{\widetilde{h}_{1}(\omega, \widetilde{u}_{1})}^{\widetilde{h}_{2}(\omega ,\widetilde{u}_{1})}\Phi_{T}^{4}(\widetilde{u}_{1},\widetilde{u}_{2},\widetilde{u}_{3})
\nonumber\\
&&\hspace*{2cm} \times \left\langle \mathcal{F}_{\widetilde{u}_{1}+\omega
,\theta_{0}}\mathcal{W}_{\omega ,\theta }
,\mathcal{F}_{\widetilde{u}_{2}-\omega
,\theta_{0}}\mathcal{W}_{\widetilde{u}_{3}-\widetilde{u}_{1} -\omega , \theta
}\right\rangle_{\mathcal{S}(\widetilde{H})}
d\widetilde{u}_{3}d\widetilde{u}_{2}d\widetilde{u}_{1}d\omega\nonumber\\
&&\leq  \frac{\mathcal{K}\pi}{T}\int_{[-\pi,\pi]^{4}}
\Phi_{4T}^{4}(u_{1},u_{2},u_{3})\nonumber\\
&& \hspace*{1cm}\times
\left\langle \mathcal{F}_{2u_{1}+\omega
,\theta_{0}}\mathcal{W}_{\omega,\theta }, \mathcal{F}_{2u_{2}-\omega
,\theta_{0}}\mathcal{W}_{-(2u_{1}+4u_{3}+\omega ),\theta }
\right\rangle_{\mathcal{S}(\widetilde{H})}d\omega  du_{1}du_{2}du_{3} \nonumber\\
&&+\frac{\mathcal{K}\pi}{T}\int_{[-\pi,\pi]^{4}}\Phi_{4T}^{4}(\widetilde{u}_{1},\widetilde{u}_{2},\widetilde{u}_{3})
\nonumber\\
&&\hspace*{1cm} \times \left\langle \mathcal{F}_{2\widetilde{u}_{1}+\omega
,\theta_{0}}\mathcal{W}_{\omega ,\theta }
,\mathcal{F}_{2\widetilde{u}_{2}-\omega
,\theta_{0}}\mathcal{W}_{4\widetilde{u}_{3}-2\widetilde{u}_{1} -\omega , \theta
}\right\rangle_{\mathcal{S}(\widetilde{H})}d\omega d\widetilde{u}_{1}d\widetilde{u}_{2}
d\widetilde{u}_{3},\nonumber\\
 \label{eqshnorm}
\end{eqnarray}
\noindent where, for $\omega \in [-\pi,\pi],$  $f_{1}(\omega )=-\pi-\omega,$  $f_{2}(\omega )=\pi-\omega ,$  $g_{1}(\omega )=-\pi+\omega ,$ $g_{2}(\omega )=\pi+\omega ,$
$h_{1}(\omega , u_{1})= -\pi-u_{1}-\omega ,$ $h_{2}(\omega , u_{1})=\pi-u_{1}-\omega,$
$\widetilde{h}_{1}(\omega , \widetilde{u}_{1})= -\pi+u_{1}+\omega ,$ $h_{2}(\omega , u_{1})=\pi+u_{1}+\omega.$
For $v_{4}=-(v_{1}+v_{2}+v_{3}),$ $v_{j}\in
[-\pi,\pi],$ $j=1,2,3,4,$   in (\ref{eqshnorm}), the multidimensional kernel  $\Phi_{T}^{4}$  of F\'ejer
type is defined as follows:'
\begin{eqnarray}
&&\Phi_{T}^{4}(v_{1},v_{2},v_{3},v_{4})=\Phi_{T}^{4}(v_{1},v_{2},v_{3})\nonumber\\
&&=\frac{1}{(2\pi)^{3}T}\sum_{t_{1},
s_{1},t_{2},s_{2}=1}^{T}\exp\left(i(t_{1}v_{1}+s_{1}v_{2}+t_{2}v_{3}+s_{2}v_{4})\right)\nonumber\\
&&\hspace*{2cm}=\frac{1}{(2\pi)^{3}T}\prod_{j=1}^{4}\frac{\sin(Tv_{j}/2)}{\sin(v_{j}/2)}\label{mfk}
\end{eqnarray}
\noindent (see, e.g., equation (6.6)  in \cite{Anh07}).

Denote in equation (\ref{eqshnorm}), for
each $k\geq 1,$ and $u_{i}\in [-\pi,\pi],$ $i=1,2,3,$
$\theta \in \Theta,$
\begin{eqnarray}&&
\hspace*{-0.5cm} G_{k1,\theta }(u_{1},u_{2},u_{3})=\int_{-\pi}^{\pi}\left\langle \mathcal{F}_{2u_{1}+\omega
,\theta_{0}}\mathcal{W}_{\omega,\theta }(\psi_{k}), \mathcal{F}_{2u_{2}-\omega
,\theta_{0}}\mathcal{W}_{-(2u_{1}+4u_{3}+\omega ),\theta }(\psi_{k})
\right\rangle_{\widetilde{H}}d\omega \nonumber\\
&& \hspace*{-0.5cm}  G_{k2,\theta }(u_{1},u_{2},u_{3})=\int_{-\pi}^{\pi}\left\langle \mathcal{F}_{2\widetilde{u}_{1}+\omega
,\theta_{0}}\mathcal{W}_{\omega ,\theta }(\psi_{k})
,\mathcal{F}_{2\widetilde{u}_{2}-\omega
,\theta_{0}}\mathcal{W}_{4\widetilde{u}_{3}-2\widetilde{u}_{1} -\omega , \theta
}(\psi_{k})\right\rangle_{\widetilde{H}}d\omega \nonumber\end{eqnarray}\begin{eqnarray}
&&
\sum_{k\geq 1}G_{k1,\theta }(u_{1},u_{2},u_{3})=\int_{-\pi}^{\pi}\left\langle \mathcal{F}_{2u_{1}+\omega
,\theta_{0}}\mathcal{W}_{\omega,\theta }, \mathcal{F}_{2u_{2}-\omega
,\theta_{0}}\mathcal{W}_{-(2u_{1}+4u_{3}+\omega ),\theta }
\right\rangle_{\mathcal{S}(\widetilde{H})}d\omega \nonumber\\
&&\sum_{k\geq 1}G_{k2,\theta }(\widetilde{u}_{1},\widetilde{u}_{2},\widetilde{u}_{3})=\int_{-\pi}^{\pi}\left\langle \mathcal{F}_{2\widetilde{u}_{1}+\omega
,\theta_{0}}\mathcal{W}_{\omega ,\theta }
,\mathcal{F}_{2\widetilde{u}_{2}-\omega
,\theta_{0}}\mathcal{W}_{4\widetilde{u}_{3}-2\widetilde{u}_{1} -\omega , \theta
}\right\rangle_{\mathcal{S}(\widetilde{H})}d\omega.\nonumber\\
\label{Gfunctbb}
\end{eqnarray}

   From equations (\ref {bonwo})  and  (\ref{bonwovv}),  for  each  $\theta \in \Theta ,$  considering $\gamma=\beta -1>0,$
   \begin{eqnarray}
  && \left\|\mathcal{F}_{\xi}\mathcal{F}_{\omega }\mathcal{W}_{\widetilde{\xi },\theta }\mathcal{W}_{\widetilde{\omega },\theta }\right\|_{\mathcal{L}(\widetilde{H})}
  \nonumber\\
  &&\leq M^{2}\left[\pi^{2(1-l(\theta ))}\right]\left[[\ln(M)]^{2}\pi^{2\gamma }M_{\widetilde{W}}^{2}+\left[L(\theta )\ln(\pi)\pi^{\gamma }M_{\widetilde{W}}\right]^{2}\right.\nonumber\\ &&+\left.\left\|\ln\left(\sigma_{\theta }^{2}\right)\right\|_{\mathcal{L}(\widetilde{H})}^{2}\left(\pi^{\gamma }M_{\widetilde{W}}\right)^{2}\right],\quad \forall \xi,\omega, \widetilde{\omega }, \widetilde{\xi } \in [-\pi,\pi].
   \label{bougk1zero2}
\end{eqnarray}

Thus,  we can apply Bounded Convergence Theorem to obtain, for each $k\geq 1,$
\begin{eqnarray}&& \lim_{u_{i}\to 0,\ i=1,2,3}
G_{k1,\theta }(u_{1},u_{2},u_{3})=\int_{-\pi}^{\pi}\lim_{u_{i}\to 0,\ i=1,2,3}\mathcal{F}_{2u_{1}+\omega
,\theta_{0}}\mathcal{F}_{2u_{2}-\omega
,\theta_{0}} \nonumber\\
&& \hspace*{4.5cm} \times \mathcal{W}^{\star }_{-(2u_{1}+4u_{3}+\omega ),\theta }
\mathcal{W}_{\omega,\theta }(\psi_{k})(\psi_{k})d\omega \nonumber\\
&&=G_{k1,\theta}(0,0,0)\nonumber\\
&& =\lim_{\widetilde{u}_{i}\to 0,\ i=1,2,3}G_{k2,\theta
}(\widetilde{u}_{1},\widetilde{u}_{2},\widetilde{u}_{3})=
\int_{-\pi}^{\pi}\lim_{\widetilde{u}_{i}\to 0,\ i=1,2,3} \mathcal{F}_{2\widetilde{u}_{1}+\omega
,\theta_{0}}\mathcal{F}_{2\widetilde{u}_{2}-\omega
,\theta_{0}}\nonumber\\
&&\hspace*{5cm}\times \mathcal{W}^{\star
}_{4\widetilde{u}_{3}-2\widetilde{u}_{1} -\omega , \theta
}\mathcal{W}_{\omega ,\theta }(\psi_{k})(\psi_{k})d\omega
\nonumber\\
&&=G_{k2, \theta }(0,0,0),
\label{bct}
\end{eqnarray}

\noindent which means that $G_{ki},$ $i=1,2,$ are continuous  at zero, and  uniform convergence holds in the limits of their convolutions with multidimensional  F\'ejer kernel. Particularly,
\begin{eqnarray}&&\lim_{T\to \infty}\int_{[-\pi,\pi]^{3}}
\hspace*{-1cm}\Phi_{4T}^{4}(v_{1},v_{2},v_{3})G_{ki,\theta}(v_{1},v_{2},v_{3})
dv_{1}dv_{2}dv_{3} = G_{ki,\theta}(0,0,0),\nonumber\\\label{zl}
\end{eqnarray}
\noindent for each $k\geq 1,$  $i=1,2,$ and $\theta \in \Theta .$

Furthermore,  the absolute integrability of the functions  $$\mathcal{G}_{1}(u_{1},u_{2},u_{3})=\sum_{k\geq 1}G_{k1,\theta }(u_{1},u_{2},u_{3}), \ \mathcal{G}_{2}(u_{1},u_{2},u_{3})=\sum_{k\geq 1}G_{k2,\theta }(u_{1},u_{2},u_{3})$$ \noindent over $[-\pi,\pi]^{3}$ holds. Specifically, under \textbf{Assumptions I--IV}, and equation (\ref{svfo}),  keeping in mind equations (\ref {bonwo})  and  (\ref{bonwovv}),  we obtain
 \begin{eqnarray}&&\int_{[-\pi,\pi]^{3}}\left|\mathcal{G}_{1}(u_{1},u_{2},u_{3})\right|\prod_{i=1}^{3}du_{i}\leq
\int_{[-\pi,\pi]^{3}}\sum_{k\geq 1}\left|G_{k1,\theta }(u_{1},u_{2},u_{3})\right|\prod_{i=1}^{3}du_{i}\nonumber\\&&= \int_{[-\pi,\pi]^{3}}\sum_{k\geq 1}\int_{\Lambda }\int_{-\pi}^{\pi}\frac{M_{2u_{1}+\omega ,\mathcal{F}}(\lambda )}{|2u_{1}+\omega |^{\alpha (\lambda ,\theta _{0})}}\frac{M_{2u_{2}-\omega ,\mathcal{F}}(\lambda )}{|2u_{2}-\omega |^{\alpha (\lambda ,\theta _{0})}}
\nonumber\\
&&\hspace*{1cm}
\times \left|\ln\left(M_{-(2u_{1}+4u_{3}+\omega ) }(\lambda
)\right)-\ln\left(\Sigma_{\theta }^{2}(\lambda )
\right)\right.\nonumber\\
&&\left. \hspace*{1.5cm}-\alpha (\lambda ,\theta ) \ln\left(
|-(2u_{1}+4u_{3}+\omega ) |\right)\right| \widetilde{W}(\lambda )|-(2u_{1}+4u_{3}+\omega ) |^{\beta
}\nonumber\\
&&\hspace*{0.5cm}
\times \left|\ln\left(M_{\omega }(\lambda
)\right)-\ln\left(\Sigma_{\theta }^{2}(\lambda )
\right)\right.\nonumber\\
&&\left. \hspace*{1.5cm}-\alpha (\lambda ,\theta ) \ln\left(
|\omega |\right)\right| \widetilde{W}(\lambda )|\omega  |^{\beta
}d\omega d\left\langle E_{\lambda }(\psi_{k}),\psi_{k}\right\rangle_{\widetilde{H}}\prod_{i=1}^{3}du_{i}\nonumber\\
&& \leq [\mathcal{H}(\theta )]^{2}\pi^{3}\int_{-\pi}^{\pi}\left\|\mathcal{F}_{\omega }\right\|_{\mathcal{S}(\widetilde{H})}^{2}d\omega <\infty.
\label{Gfunctcc}
\end{eqnarray}
\noindent  Similarly, we can prove that $\mathcal{G}_{2}(u_{1},u_{2},u_{3})\in L^{1}([-\pi,\pi]^{3}).$
Thus, the following limits are obtained from the  convolution of functions $\mathcal{G}_{1}(u_{1}, u_{2}, u_{3}),$ and $\mathcal{G}_{2}(u_{1}, u_{2}, u_{3})$ with  F\'ejer kernel in (\ref{eqshnorm}):
\begin{eqnarray} && \lim_{T\to \infty}
\frac{\mathcal{K}\pi}{T}\int_{[-\pi,\pi]^{4}}
\Phi_{4T}^{4}(u_{1},u_{2},u_{3})\nonumber\\
&& \hspace*{1cm}\times
\left\langle \mathcal{F}_{2u_{1}+\omega
,\theta_{0}}\mathcal{W}_{\omega,\theta }, \mathcal{F}_{2u_{2}-\omega
,\theta_{0}}\mathcal{W}_{-(2u_{1}+4u_{3}+\omega ),\theta }
\right\rangle_{\mathcal{S}(\widetilde{H})}d\omega  du_{1}du_{2}du_{3} \nonumber\\
&&+\lim_{T\to \infty}\frac{\mathcal{K}\pi}{T}\int_{[-\pi,\pi]^{4}}\Phi_{4T}^{4}(\widetilde{u}_{1},\widetilde{u}_{2},\widetilde{u}_{3})
\nonumber\\
&&\hspace*{1cm} \times \left\langle \mathcal{F}_{2\widetilde{u}_{1}+\omega
,\theta_{0}}\mathcal{W}_{\omega ,\theta }
,\mathcal{F}_{2\widetilde{u}_{2}-\omega
,\theta_{0}}\mathcal{W}_{4\widetilde{u}_{3}-2\widetilde{u}_{1} -\omega , \theta
}\right\rangle_{\mathcal{S}(\widetilde{H})}d\omega d\widetilde{u}_{1}d\widetilde{u}_{2}
d\widetilde{u}_{3}\nonumber\\
&&=\lim_{T\to \infty}\frac{\mathcal{K}\pi}{T}\left[\mathcal{G}_{1}(0,0,0)+\mathcal{G}_{2}(0,0,0)\right].
\label{Gfuncthh}
\end{eqnarray}

From equations (\ref{eqshnorm})--(\ref{Gfuncthh}),
 as $T\to \infty,$
$$E\left\| \int_{-\pi}^{\pi}\left[p_{\omega
}^{(T)}-\mathcal{F}^{(T)}_{\omega,\theta_{0}}\right]
\mathcal{W}_{\omega,\theta
}d\omega\right\|_{\mathcal{S}(\widetilde{H})}^{2}=
\mathcal{O}\left(\frac{1}{T}\right).$$

Applying Jensen's inequality,
\begin{eqnarray}&&
E\left\| \int_{-\pi}^{\pi}\left[p_{\omega
}^{(T)}-\mathcal{F}^{(T)}_{\omega,\theta_{0}}\right]
\mathcal{W}_{\omega,\theta
}d\omega\right\|_{\mathcal{S}(\widetilde{H})}\nonumber\\
&&\leq
\sqrt{E\left\| \int_{-\pi}^{\pi}\left[p_{\omega
}^{(T)}-\mathcal{F}^{(T)}_{\omega,\theta_{0}}\right]
\mathcal{W}_{\omega,\theta
}d\omega\right\|_{\mathcal{S}(\widetilde{H})}^{2}}\to 0,\quad T\to \infty.\nonumber\\
\label{eqji}
\end{eqnarray}
 From (\ref{limuk})  and (\ref{eqji}), applying triangle inequality, equation (\ref{eqtlimits1}) holds. In particular,

 \begin{equation}\left\|U_{T,\theta}-U_{\theta }
 \right\|_{\mathcal{S}(\widetilde{H})}\to_{P} 0,
 \quad T\to \infty,\quad  \forall \theta \in \Theta.\label{eqcprob}
 \end{equation}

  Therefore, for each $\theta \in \Theta ,$ as $T\to \infty,$
   \begin{eqnarray}&&\left\|
   U_{T,\theta}-U_{T,\theta_{0}}-\mathcal{K}(\theta_{0},\theta )
   \right\|_{\mathcal{S}(\widetilde{H})}\nonumber\\
 &&\hspace*{-1cm} = \left[\sum_{k,l\geq 1}\left|[U_{T,\theta}-U_{T,
 \theta_{0}}](\psi_{k})(\psi_{l})
 -[\mathcal{K}(\theta_{0},\theta )](\psi_{k})(\psi_{l})\right|^{2}\right]^{1/2}
  \to_{P} 0,\nonumber\\
  \label{clfkld}
  \end{eqnarray}
\noindent implying that, as $T\to \infty,$
\begin{equation}\sup_{k\geq 1}\left|[U_{T,\theta}-U_{T,\theta_{0}}](\psi_{k})(\psi_{k})
-[\mathcal{K}(\theta_{0},\theta )](\psi_{k})(\psi_{k})\right|
\to_{P} 0. \label{convtunik}\end{equation} \noindent From reverse triangle inequality,
denoting $L_{T}(\theta )=\sup_{k\geq 1}\left|[U_{T,\theta
}-U_{T,\theta_{0}}](\psi_{k})(\psi_{k})\right|$ and $\mathcal{L}(\theta
)=\sup_{k\geq 1}\left|[\mathcal{K}(\theta_{0},\theta
)](\psi_{k})(\psi_{k})\right|,$
  \begin{equation}L_{T}(\theta ) \to_{P}
\mathcal{L}(\theta ),  \quad T\to \infty,\quad \forall \theta \in \Theta.
\label{convunifemplf}\end{equation}
  From equations (\ref{eqjibb})--(\ref{mfethetabb}),
       \begin{eqnarray}&& \mathcal{L}(\theta )=
   \sup_{k\geq 1}[\mathcal{K}(\theta_{0},\theta )](\psi_{k})(\psi_{k})>0,
   \quad \theta\neq \theta_{0}\nonumber\\
     &&\theta_{0}=\mbox{arg} \ \min_{\theta \in \Theta }\mathcal{L}(\theta )=
 \mbox{arg} \
 \min_{\theta \in \Theta }\sup_{k\geq 1}
 [\mathcal{K}(\theta_{0}, \theta )](\psi_{k})(\psi_{k}),\label{mcf}
\end{eqnarray}
\noindent for any orthonormal basis
$\{\psi_{k},\ k\geq 1\}$ of $\widetilde{H}.$

To prove the consistency of the estimator $\widehat{\theta }_{T}$ in
(\ref{mfetheta}), we first show that the convergence
(\ref{convunifemplf}) holds uniformly in $\theta \in \Theta .$
Specifically, for any $\theta_{1},\theta_{2}\in \Theta,$ from
equation (\ref{standarisdo}),
 considering  triangle inequality, and the fact that $p_{\omega }^{(T)}$
 and $\mathcal{W}_{\omega }$ are non--negative operators
 for every $\omega \in [-\pi,\pi],$ we obtain,  for each $k\geq 1,$
\begin{eqnarray}
&&\left|U_{T,\theta_{1}}-U_{T,\theta_{2}}(\psi_{k})(\psi_{k})
\right|\nonumber\\&&\leq \int_{-\pi}^{\pi}\left|p_{\omega }^{(T)}
\ln\left(\Upsilon_{\omega ,\theta_{2} }\Upsilon_{\omega
,\theta_{1}}^{-1}\right)\mathcal{W}_{\omega
}(\psi_{k})(\psi_{k})\right|d\omega
\nonumber\\
&&= \int_{-\pi}^{\pi}\left| \ln\left( \sigma^{2}_{\theta_{1}}[\sigma^{2}_{\theta_{2}}]^{-1}\right)
+\left(\mathcal{A}_{\theta_{1}}-\mathcal{A}_{\theta_{2}}\right)\ln
\left(|\omega |\right)\right|\nonumber\\&& \hspace*{3cm}\times
\left|p_{\omega }^{(T)}\mathcal{W}_{\omega
}(\psi_{k})(\psi_{k})\right|d\omega
\nonumber\\
&&\leq \left\| \ln\left(\sigma^{2}_{\theta_{1}}[\sigma^{2}_{\theta_{2}}]^{-1}\right)\right\|_{\mathcal{L}(\widetilde{H})}
 \int_{-\pi}^{\pi}p_{\omega }^{(T)}
 \mathcal{W}_{\omega }(\psi_{k})(\psi_{k})
 d\omega \nonumber\\
&&+\left\|\mathcal{A}_{\theta_{1}}-
\mathcal{A}_{\theta_{2}}\right\|_{\mathcal{L}(\widetilde{H})}
\int_{-\pi}^{\pi} \left|\ln \left(|\omega |\right)\right|  p_{\omega
}^{(T)}\mathcal{W}_{\omega }(\psi_{k})(\psi_{k})d\omega . \label{mc}
\end{eqnarray}

From (\ref{mc}), to prove the convergence (\ref{convunifemplf})
holds  uniformly in $\theta \in \Theta ,$ we only need to show that,
for any $k\geq 1,$
\begin{eqnarray}
&&\int_{-\pi}^{\pi }p_{\omega }^{(T)}\mathcal{W}_{\omega}
(\psi_{k})(\psi_{k})d\omega =\mathcal{O}_{P}(1),\quad T\to \infty\label{pconv0}\\
&&\int_{-\pi}^{\pi} \left|\ln \left(|\omega |\right)\right|
p_{\omega }^{(T)}\mathcal{W}_{\omega}(\psi_{k})(\psi_{k})d\omega
=\mathcal{O}_{P}(1),\quad T\to \infty \label{pconv}
\end{eqnarray}
\noindent (see Theorems 21.9 and 21.10 in \cite{Davidson94}). Note
that,  for $k\geq 1,$
\begin{eqnarray}
&&\sigma_{\theta _{0}}^{2}(\psi_{k})(\psi_{k})=\int_{-\pi}^{\pi
}\mathcal{F}_{\omega ,\theta _{0}}\mathcal{W}_{\omega}
(\psi_{k})(\psi_{k})d\omega \leq \|\sigma_{\theta
_{0}}^{2}\|_{\mathcal{L}(\widetilde{H})}
< \infty\label{pconv0vs}\\
&&\int_{-\pi}^{\pi}\left|\ln \left(|\omega |\right)\right|
\mathcal{F}_{\omega ,
\theta_{0}}\mathcal{W}_{\omega}(\psi_{k})(\psi_{k})d\omega \leq
2\pi\sup_{(\omega ,\lambda )\in [-\pi,\pi]\times \Lambda }
|\ln (|\omega |)|/|\omega |^{\alpha (\lambda ,\theta_{0})-\beta }\nonumber\\
&& \hspace*{1cm} \times \sup_{\omega\in [-\pi,\pi]}\|\widetilde{W}\mathcal{M}_{\omega ,\mathcal{F}
}\|_{\mathcal{L}(\widetilde{H})}<\infty,\label{pconvss}
\end{eqnarray}
\noindent where, for $\beta >1,$
$$\mathcal{C}=2\pi\sup_{(\omega ,\lambda )\in [-\pi,\pi]\times \Lambda }
|\ln (|\omega |)|/|\omega |^{\alpha (\lambda ,\theta_{0})-\beta }<\infty.
$$
From Theorem \ref{pr1},  as $T\to
\infty,$
\begin{eqnarray}\left\|
\int_{-\pi}^{\pi} \left[E\left[p_{\omega
}^{(T)}\right]-\mathcal{F}_{\omega ,\theta _{0}} \right]
\mathcal{W}_{\omega}d\omega \right\|_{\mathcal{S}(\widetilde{H})}
\to 0\label{cetambi}\\
\left\|\int_{-\pi}^{\pi}|\ln (|\omega |)|\left[E\left[p_{\omega
}^{(T)} \right]-\mathcal{F}_{\omega ,\theta _{0}}\right]
\mathcal{W}_{\omega}d\omega \right\|_{\mathcal{S}(\widetilde{H})}\to
0.\label{convzerex}
\end{eqnarray}

In a similar way to equations (\ref{eqshnorm})--(\ref{eqcprob}), it
can also be proved that  \begin{eqnarray} &&E\left\|\int_{-\pi}^{\pi
}\left[p_{\omega }^{(T)}- \mathcal{F}_{\omega ,\theta
_{0}}^{(T)}\right]\mathcal{W}_{\omega}d\omega
\right\|^{2}_{\mathcal{S}(\widetilde{H})} \to 0,\  T\to \infty
\label{eqconvshuc1}
\\
&&E\left\|\int_{-\pi}^{\pi }|\ln (|\omega |)|\left[p_{\omega
}^{(T)}- \mathcal{F}_{\omega ,\theta
_{0}}^{(T)}\right]\mathcal{W}_{\omega}d\omega
\right\|^{2}_{\mathcal{S}(\widetilde{H})} \to 0, \ T\to \infty .
\label{eqconvshuc2}
\end{eqnarray}
From
equations (\ref{pconv0vs})--(\ref{eqconvshuc2}),
  as $T\to \infty,$ \begin{eqnarray}
&&\left\|\int_{-\pi}^{\pi }\left[p_{\omega }^{(T)}-
\mathcal{F}_{\omega ,\theta _{0}}\right]\mathcal{W}_{\omega}d\omega
\right\|_{\mathcal{S}(\widetilde{H})} \to_{P} 0\label{eqconvshuc1vv}
\\
&&\left\|\int_{-\pi}^{\pi }|\ln (|\omega |) |\left[p_{\omega
}^{(T)}-\mathcal{F}_{\omega ,\theta _{0}}\right]
\mathcal{W}_{\omega}d\omega \right\|_{\mathcal{S}(\widetilde{H})}
\to_{P} 0. \label{eqconvshuc2vv}
\end{eqnarray}

\noindent From (\ref{eqconvshuc1vv})--(\ref{eqconvshuc2vv}),
equations (\ref{pconv0}) and (\ref{pconv})  are satisfied  uniformly
in $k\geq 1.$ Thus, (\ref{convunifemplf}) holds uniformly in $\theta
\in \Theta .$

To prove $\widehat{\theta }_{T}$ is weakly consistent, consider that
$\widehat{\theta }_{T}$ does not converge in probability to
$\theta_{0}.$ Hence, there exists a subsequence $\{\widehat{\theta
}_{T_{m}},\ m\in \mathbb{N}\}$ such that $\widehat{\theta
}_{T_{m}}\to_{P}\theta^{\prime }\neq \theta_{0},$ as $T_{m}\to
\infty,$ when $m\to \infty.$  From (\ref{mcf}),
 for  $\tau >0$ satisfying $0<\nu< \mathcal{L}(\theta^{\prime })-\tau ,$
 for certain $\nu>0,$ applying uniform convergence  in $\theta \in \Theta ,$ in equation   (\ref{convunifemplf}),    there exists $m_{0}$ such that for $m\geq m_{0},$
\begin{equation}
P\left[\inf_{l\geq m}L_{T_{l}}(\widehat{\theta }_{T_{l}})\geq
\mathcal{L}(\theta ^{\prime })-\tau>\nu >0 \right]\geq p_{0}>1/2.
\label{eqfpi}
\end{equation}

From  equations (\ref{clfkld}), (\ref{convtunik}) and (\ref{mcf}),
for $T$ sufficiently large, $$U_{T,\theta
}-U_{T,\theta_{0}}(\psi_{k})(\psi_{k})\geq 0,\quad \forall k\geq
1.$$ \noindent Then, from definition of the estimator
  $\widehat{\theta }_{T}$  in (\ref{mfetheta}), and  uniform
  convergence in probability in (\ref{convunifemplf}),
    that also holds in the  $\mathcal{S}(\widetilde{H})$ norm
  (see equations (\ref{eqconvshuc1vv})--(\ref{eqconvshuc2vv})),
  there exists $m_{0}^{\star}$ such that for $m\geq m_{0}^{\star},$
\begin{equation}
P\left[\sup_{l\geq m}L_{T_{l}}(\widehat{\theta }_{T_{l}})\leq
\inf_{\theta \in \Theta } \mathcal{L}(\theta )=\mathcal{L}(\theta_{0})=0\right]\geq
p_{0}>1/2, \label{eqfpi2}
\end{equation}
\noindent which, in particular, implies
\begin{equation}
P\left[\inf_{l\geq m}L_{T_{l}}(\widehat{\theta }_{T_{l}})\leq
\inf_{\theta \in \Theta } \mathcal{L}(\theta )=\mathcal{L}(\theta_{0})=0\right]\geq
p_{0}>1/2. \label{eqfpi3}
\end{equation}

For $m\geq \max\{m_{0},m_{0}^{\star}\},$  equations
(\ref{eqfpi})--(\ref{eqfpi3}) lead to a contradiction. Thus,
$\widehat{\theta }_{T}\to_{P}\theta_{0},$ as $T\to \infty.$

\end{proof}
\begin{remark} The
multifractionally integrated functional autoregressive
moving averages process family introduced in Section \ref{SFIFLTS}
satisfies the conditions assumed in Theorem \ref{theslrd}, for a
suitable choice of the  polynomial sequence $\left\{
\Phi_{p,l},\ \Psi_{q,l}, \  l\geq 1\right\}.$
\end{remark}

\section{Final comments}
\label{conclus}
The spectral analysis of SRD functional time series has been currently achieved  in several papers. Particularly, in the Introduction, we have referred  to
the pioneer contribution in \cite{Panaretos13}. This paper constitutes a first attempt  in the spectral analysis of stationary functional time series beyond the SRD condition.  Specifically,
 this paper applies spectral theory of self--adjoint operators on a separable Hilbert space to
characterize LRD in functional time series in the spectral domain, under \textbf{Assumptions I--IV} (see Proposition \ref{prlrd}). As special cases, multifractionally integrated functional ARMA processes  are considered  (see  Section \ref{SFIFLTS}). Their  tapered continuous version in the spectral domain is also analyzed in Section \ref{CSCS}. This second example allows the implementation of parametric estimation  techniques in the functional spectral domain, from the  discrete sampling in time of the solution to models introduced in \cite{AnhLeonenkoa};  \cite{AnhLeonenkob};  \cite{Kelbert05}.  Our main results, Theorems \ref{pr1}  and \ref{theslrd}, respectively provide the convergence   to zero  in $\mathcal{S}(\widetilde{H})$ norm of the  bias of the integrated periodogram operator,  and the weak consistent estimation of the LRD operator, in a parametric framework in the spectral domain. Note that Theorem  \ref{pr1} holds  beyond the linear and Gaussian case, under our  LRD setting, while   Theorem  \ref{theslrd} is proved under a   LRD Gaussian scenario.

%
%

\section*{Acknowledgements}
This work has been supported in part by project PGC2018-099549-B-I00
of the Ministerio de Ciencia, Innovaci\'on y Universidades, Spain
(co-funded with FEDER funds).
This work is also supported in part by the IMAG–Maria de Maeztu grant \\ CEX2020-001105-M / AEI / 10.13039/501100011033

We would like to thank  Professors Antonio Cuevas and Daniel Pe\~na
for their helpful comments and suggestions that have  contributed to
the improvement of the present paper in an important way.

\end{document}